\documentclass[12pt]{amsart}
\usepackage{amssymb}
\usepackage{verbatim}
\usepackage{hyperref}
\usepackage{color}

\begin{document}

% define theorem environments
\newtheorem{theorem}{Theorem}    %[section]
\newtheorem{proposition}[theorem]{Proposition}
\newtheorem{conjecture}[theorem]{Conjecture}
\def\theconjecture{\unskip}
\newtheorem{corollary}[theorem]{Corollary}
\newtheorem{lemma}[theorem]{Lemma}
\newtheorem{sublemma}[theorem]{Sublemma}
\newtheorem{fact}[theorem]{Fact}
\newtheorem{observation}[theorem]{Observation}
\theoremstyle{definition}
\newtheorem{definition}[theorem]{Definition}
\newtheorem{notation}[theorem]{Notation}
\newtheorem{remark}[theorem]{Remark}
\newtheorem{question}[theorem]{Question}
\newtheorem{questions}[theorem]{Questions}
\newtheorem{example}[theorem]{Example}
\newtheorem{problem}[theorem]{Problem}
\newtheorem{exercise}[theorem]{Exercise}

\numberwithin{theorem}{section}
\numberwithin{equation}{section}

\def\reals{{\mathbb R}}
\def\torus{{\mathbb T}}
\def\integers{{\mathbb Z}}
\def\naturals{{\mathbb N}}
\def\complex{{\mathbb C}\/}
\def\distance{\operatorname{distance}\,}
\def\support{\operatorname{support}\,}
\def\dist{\operatorname{dist}\,}
\def\Span{\operatorname{span}\,}
\def\degree{\operatorname{degree}\,}
\def\dim{\operatorname{dim}\,}
\def\codim{\operatorname{codim}}
\def\trace{\operatorname{trace\,}}
\def\Span{\operatorname{span}\,}
\def\dimension{\operatorname{dimension}\,}
\def\codimension{\operatorname{codimension}\,}
\def\nullspace{\scriptk}
\def\kernel{\operatorname{Ker}}
\def\Re{\operatorname{Re\,} }
\def\Im{\operatorname{Im\,} }
\def\eps{\varepsilon}
\def\lt{L^2}
\def\diver{\operatorname{div}}
\def\curl{\operatorname{curl}}
\def\expect{\mathbb E}
\def\bull{$\bullet$\ }
\def\det{\operatorname{det}}
\def\Det{\operatorname{Det}}

\newcommand{\norm}[1]{ \|  #1 \|}
\newcommand{\Norm}[1]{ \Big\|  #1 \Big\| }
\newcommand{\set}[1]{ \left\{ #1 \right\} }
\def\one{{\mathbf 1}}
\newcommand{\modulo}[2]{[#1]_{#2}}
\newcommand{\abr}[1]{ \langle  #1 \rangle}

\renewcommand{\qed}{\hfill \mbox{\raggedright \rule{0.1in}{0.1in}}}

\def\bp{\mathbf p}
\def\bff{\mathbf f}
\def\bg{\mathbf g}

\def\b1{\mathbf 1}
\def\bi{\mathbf i}
\def\bj{\mathbf j}
\def\bk{\mathbf k}

\def\bu{\mathbf u}
\def\bv{\mathbf v}
\def\bw{\mathbf w}
\def\bx{\mathbf x}
\def\by{\mathbf y}

\def\scriptf{{\mathcal F}}
\def\scriptg{{\mathcal G}}
\def\scriptm{{\mathcal M}}
\def\scriptb{{\mathcal B}}
\def\scriptc{{\mathcal C}}
\def\scriptt{{\mathcal T}}
\def\scripti{{\mathcal I}}
\def\scripte{{\mathcal E}}
\def\scriptv{{\mathcal V}}
\def\scriptw{{\mathcal W}}
\def\scriptu{{\mathcal U}}
\def\scriptS{{\mathcal S}}
\def\scripta{{\mathcal A}}
\def\scriptr{{\mathcal R}}
\def\scripto{{\mathcal O}}
\def\scripth{{\mathcal H}}
\def\scriptd{{\mathcal D}}
\def\scriptl{{\mathcal L}}
\def\scriptn{{\mathcal N}}
\def\scriptp{{\mathcal P}}
\def\scriptk{{\mathcal K}}
\def\scriptP{{\mathcal P}}
\def\scriptj{{\mathcal J}}
\def\frakv{{\mathfrak V}}
\def\frakG{{\mathfrak G}}
\def\frakA{{\mathfrak A}}

\author{Marcos Charalambides}
\address{
        Marcos Charalambides\\
        Department of Mathematics\\
        University of California \\
        Berkeley, CA 94720-3840, USA}
\email{marcos@math.berkeley.edu}

\thanks{Research supported in part by NSF grant DMS-0901569.}

\date{February 2nd, 2014}

\title
[Distinct distances on curves via rigidity]
{Distinct distances on curves via rigidity}

\maketitle
\begin{abstract}
It is shown that $N$ points on a real algebraic curve of degree $n$ in $\reals^d$ always determine $\gtrsim_{n,d}N^{1+\frac{1}{4}}$ distinct distances, unless the curve is a straight line or the closed geodesic of a flat torus. In the latter case, there are arrangements of $N$ points which determine $\lesssim N$ distinct distances. The method may be applied to other quantities of interest to obtain analogous exponent gaps. An important step in the proof involves understanding the structural rigidity of certain frameworks on curves.
\end{abstract}

\section{Introduction}

Let $P\subset\reals^2$ be a finite set. Consider the set \[\Delta(P):=\{\norm{p-q}\,\,|\,\,p,q\in P\}\] of distances deteremined by $P$. A famous problem, posed by Erd\"{o}s \cite{erdosdist}, is to determine a sharp asymptotic lower bound on the cardinality $|\Delta(P)|$ of this set as a function of the cardinality $|P|$ of the set $P$. 

\begin{conjecture}[Erd\H{o}s]
Let $P\subset\reals^2$ be a finite subset. Then, \[|\Delta(P)|\gtrsim\frac{|P|}{\sqrt{\log |P|}}.\]
\end{conjecture}

Recently, Guth and Katz have proven the following celebrated (almost-sharp) result.

\begin{theorem}[Guth-Katz \cite{guthkatz}]\label{thm:guthkatz}
Let $P\subset\reals^2$ be a finite subset. Then, \[|\Delta(P)|\gtrsim\frac{|P|}{\log |P|}.\]
\end{theorem}

\begin{remark}
The exponent $1$ of $|P|$ in the lower bound is sharp; for example a set of $N$ equally spaced points on a circle or a straight line determines $\le N$ distinct distances.
\end{remark}

When considering the inverse problem of describing arrangements of points which determine few distinct distances, one question which arises is whether these arrangements have algebro-geometric structure. In this article, we look at whether arrangements of points in $\reals^2$ and, more generally, $\reals^d$ which are known to lie on an algebraic curve of fixed degree can determine too few distinct distances. We explore a link between the algebraic geometry of the problem and the structural rigidity of certain frameworks on the curve and, with this interpretation, we are able to show that, unless the curve has a very specific form which we can describe explicitly, a finite set of points lying on the curve cannot determine too few distinct distances.

\begin{remark}\label{remark:distdegen}
Even without any additional assumptions, finite subsets of algebraic curves cannot determine too few distances. Indeed, let $\Gamma$ be an algebraic curve of degree $n$ in $\reals^2$. Then $\Gamma$ intersects any circle which does not contain $\Gamma$ in at most $2n$ points. Consequently, every point in $P$ determines at least $\frac{|P|-1}{2n}$ distinct distances with the other points of $P$. (The author thanks the second anonymous referee for this argument.)
\end{remark}

\subsection{Helices}

We now describe a class of real analytic curves supporting finite subsets which determine few distinct distances.

\begin{definition}[Generalized helix]\label{def:helix}
Let $d>0$, $k,l\ge 0$ and $l+2k\le d$. Let $A$ be a real invertible skew-symmetric $2k\times 2k$ matrix, $v\in\reals^{2k}$ and $w\in\reals^l$. A \emph{generalized helix} is a real-analytic curve $\Gamma$ in $\reals^d$ parametrized by $\gamma:I\to\reals^d$ for a non-empty open interval $I\subset\reals$ which, up to rigid motions, is given by \[\gamma(t)=(\exp(At)v, tw, 0)\in\reals^{2k}\times\reals^l\times\reals^{d-l-2k}=\reals^d.\]
\end{definition}

Generalized helices with $l=0$ are (up to rigid motions) the geodesics of a $k$-dimensional flat torus parametrized by \[(u_1, \ldots, u_k)\mapsto(\alpha_1\cos u_1, \alpha_1\sin u_1, \ldots, \alpha_k\cos u_k, \alpha_k\sin u_k)\in\reals^{2k}\subset\reals^d\] for some $\alpha_1, \ldots, \alpha_k >0$. The flat torus is the embedded $k$-dimensional submanifold of $\reals^{2k}$ obtained by taking a $k$-fold product of circles $S^1\subset\reals^2$.

A generalized helix is a real algebraic curve if and only if either $k=0$ and $l>0$ (in other words, it is a straight line) or, alternatively, $k>0$, $l=0$ and the curve is a geodesic of a $k$-dimensional flat torus which is closed. This means that it has a parametrization of the form \[
\gamma(t)=(\alpha_1\cos \lambda_1t, \alpha_1\sin \lambda_1 t, \ldots, \alpha_k\cos \lambda_kt, \alpha_k\sin \lambda_k t).
\]
with each ratio $\lambda_j/\lambda_i$ rational; see Lemma~\ref{lemma:algebraichelix}. We will refer to such a curve as an \emph{algebraic helix}.

\begin{remark}
In $\reals^2$ (and $\reals^3$), algebraic helices are straight lines and circles.
\end{remark}

Algebraic helices, and in fact generalized helices, support subsets which determine few distinct distances.

\begin{theorem}
Let $\Gamma$ be a generalized helix in $\reals^d$. Then, for any integer $N>0$ there exists a finite subset $P$ of $\Gamma$ such that $|P|=N$ and the number of distinct distances determined by $P$ is $\lesssim N$.
\end{theorem}

\begin{proof}
Let $\Gamma$ be given by \[\gamma(t)=(\exp(At)v, tw, 0).\] Let $Q\subset\reals$ be any finite arithmetic progression of cardinality $N$. Then for $x, y\in Q$, 
\begin{eqnarray}
\norm{\gamma(x)-\gamma(y)}^2&=&\norm{(\exp(Ax)-\exp(Ay))v}^2+(x-y)^2\norm{w}^2\nonumber\\
&=&(x-y)^2\big(\norm{\exp(A\xi)Av}^2+\norm{w}^2\big)\nonumber
\end{eqnarray} for some $\xi$ in the interval $[\min\{x,y\}, \max\{x, y\}]$. Since $A$ is skew-symmetric, $\exp(A\xi)$ is orthogonal and
\begin{equation}
\norm{\exp(A\xi)Av}^2=\norm{Av}^2.\nonumber
\end{equation}
Hence
\[
\norm{\gamma(x)-\gamma(y)}=|x-y|(\norm{Av}^2+\norm{w}^2)^{\frac{1}{2}}.
\]

Consequently, the set of pairwise distances \[\{\norm{\gamma(x)-\gamma(y)}\,\,|\,\,x,y\in Q\}\] determined by image $P=\gamma(Q)$ of $Q$ under $\gamma$ has cardinality $\le|Q-Q|\lesssim|Q|=N$.
\end{proof}

\subsection{Main results}

Our main result is that for real algebraic curves (see Section~\ref{sect:curves} for a precise definition) which are not generalized helices, there is an exponent gap and the set of pairwise distances determined by a finite subset $P$ has cardinality $\gtrsim|P|^{1+\delta}$ for some $\delta>0$. We obtain $\delta=\frac{1}{4}$ in the proof below, although we do not believe this is optimal.

\begin{theorem}\label{thm:algdist}
Suppose that $\Gamma\subset\reals^d$ is a real algebraic curve of degree $m$. Let $P\subset\Gamma$ be a finite subset. If no irreducible component of $\Gamma$ is an algebraic helix, the number of distinct distances determined by $P$ is $\gtrsim_{m,d}|P|^{1+\frac{1}{4}}$. 
\end{theorem}

\begin{remark} A few weeks after an initial preprint of this paper was released, Pach and de Zeeuw \cite{pachdezeeuw} improved the exponent for the special case when the curve is embedded in the plane (i.e. when $d=2$) to $1+\frac{1}{3}$ using a more direct algebraic argument. Their argument is simpler and shorter than our approach for that special case but does not seem to readily generalize to higher ambient dimensions and does not explore the link to structural rigidity which we look at here. 
\end{remark}

The method can also be applied to quantities of interest other than the number of pairwise distances determined by $P$. To illustrate this, we will also show how to obtain an analogous result for the number of distinct areas of triangles in $\reals^2$ determined by pairs of points in a finite set $P$ and a fixed apex $v$ which does not lie on $\Gamma$. 

\begin{remark}
In \cite{iosevichrnrudnev}, Iosevich, Roche-Newton and Rudnev proved the analogue of Theorem~\ref{thm:guthkatz} for this quantity: they show that a finite non-collinear set of points $P$ in the plane determines $\gtrsim |P|/\log|P|$ distinct areas of triangles with one vertex at the origin. Areas of triangles without the restriction that one vertex is fixed have also been studied by Pinchasi \cite{pinchasi} who gave an exact bound in this case.
\end{remark}

\begin{theorem}\label{thm:algareas}
Suppose that $\Gamma\subset\reals^2$ is a real algebraic curve of degree $m$ and $v\in\reals^2$ is not on $\Gamma$. Let $P\subset\Gamma$ be a finite subset. If no irreducible component of $\Gamma$ is a straight line or an ellipse or hyperbola centered at $v$, the number of distinct areas of triangles with vertices at $v$, $p$ and $q$ for pairs of points $p,q\in P$ is $\gtrsim_m|P|^{1+\frac{1}{4}}$. 
\end{theorem}

While we do not attempt to state a fully general theorem in this paper, we do prove the analogue of these results to a large class of quantities; see Theorem~\ref{thm:general}.

\begin{remark}
Similarly to Remark~\ref{remark:distdegen}, the number of areas of triangles is $\gtrsim_{m}|P|$ for any irreducible curve $\Gamma$. Taking equally-spaced points on a circle centered at $v$ or on a straight line shows that for these two classes of curves there are finite subsets $P$ which determine $\lesssim |P|$ distinct areas of triangles. Similarly, for any geometric progression $Q$, taking the points $P=\{(q+v_1,q^{-1}+v_2)\,\,|\,\, q\in Q\}$ on the rectangular hyperbola given by $(X-v_1)(Y-v_2)=1$ gives an example of a $P$ on this curve which determines $\lesssim |P|$ distinct areas of triangles. Since any affine transformation which fixes $v$ will preserve the number of distinct areas of triangles, it follows that for any ellipse or hyperbola centered at $v$ or any straight line and any integer $N>0$, there are examples of finite subsets consisting of $N$ points which determine only $\lesssim N$ distinct areas of triangles.  
\end{remark}

Before we begin the proof of Theorem~\ref{thm:algdist}, it will be necessary to review some basic results from algebraic geometry and a introduce some language from the theory of structural rigidity. We have collected these prerequisites in Section~\ref{sect:prereq}. In the final section, we discuss some links to other results in the literature.

\subsection{Outline of proof}

The proof itself turns out to be technical, even though the argument is quite elementary. As a consequence, we give a brief and informal expository outline to help navigate the reader. The proof itself begins in Section~\ref{sect:gen}. 

Consider a real algebraic curve $\Gamma\subset\reals^d$. We define exactly what we mean by this in Section~\ref{sect:curves}. 

The first step to proving Theorem~\ref{thm:algdist} is to follow a method of Elekes (see Proposition~\ref{prop:xixi}) to reduce the question to checking whether certain \emph{planar} algebraic curves intersect a lot; if they do not, then the original curve $\Gamma$ cannot support finite subsets determining few distinct distances. This is done in Section~\ref{sect:gen}.

Once this reduction is performed, the main general step concerns showing that the curves constructed do not intersect too much; this is done in Section~\ref{sect:admiss}. The main novelty in our technique involves showing that if this intersection property fails then the curve enjoys a very restrictive structural property: loosely speaking, it is possible to move any triangle with vertices on the curve along the curve while keeping its edge lengths fixed (more precisely, we show that a local version of this property must hold). This result furnishes a link between distinct distance results and the theory of structural rigidity. So as to not disrupt the flow of the main argument, the proof of the rigidity results is done in Section~\ref{sect:rig}.

The argument so far applies more generally to other quantities of interest on curves, not just Euclidean distance. While we do not explore a full generalization in this paper, we give a satisfactory generalization of the above result to a certain class of symmetric algebraic quantities $D(p,q)$ between pairs of points $p,q\in\Gamma$. The same argument shows that the curves which support finite subsets $P$ determining only a few distinct values $D(p,q)$ as $p,q$ vary over $P$ enjoy an analogous structural property: it is possible to move any triangle with vertices $p,q,r$ on the curve along the curve while keeping the values of $D(p,q)$, $D(q,r)$ and $D(r,p)$ fixed.

For the particular case where $D$ is Euclidean distance and the original distinct distance problem, the proof of Theorem~\ref{thm:algdist} is completed by characterizing real algebraic curves which have the property that triangles may be moved along them while preserving edge lengths. This is done in Section~\ref{sect:degen}. The key step is to show that the property forces the norm of every derivative of (a parametrization of) the curve to be constant; this is done by using a finite difference approximation to link the structural rigidity of points along the curve to a statement about derivatives (see Lemma~\ref{lemma:constcurv}). Finally, we use a result of D'Angelo and Tyson \cite{helices} which states that real analytic curves with all derivatives of constant norm are necessarily generalized helices.

Due to a technical reason which arises in the proof, it is more convenient to work with a simpler class of curves (see Definition~\ref{def:simple}); the reduction to this case is performed in Section~\ref{sect:simplicity}. 

\section{Technical prerequisites}\label{sect:prereq}

\subsection{Preliminaries}

Given a finite set $P$, we will denote its cardinality by $|P|$. We will write $P^{2*}$ for the restricted Cartesian product, \[P^{2*}=\{(p,q)\,\,|\,\,p,q\in P, \,p\ne q\}.\]

If $f_1$ and $f_2$ are non-negative functions $\mathbb{N}\to\reals_{\ge 0}$ and $\nu_1, \ldots, \nu_r$ is a list of parameters, we will write $f_1\lesssim_{\nu_1, \ldots, \nu_r}f_2$ to mean that there exist $n_0=n_0(\nu_1,\ldots,\nu_r)\in\mathbb{N}$ and a real function $C=C(\nu_1, \ldots,\nu_r)>0$ depending only on $\nu_1, \ldots, \nu_r$ such that for all $n>n_0$, it follows that $f_1(n)\le Cf_2(n)$. We will also write $f_1\gtrsim_{\nu_1, \ldots,\nu_r}f_2$ to mean $f_2\lesssim_{\nu_1, \ldots,\nu_r}f_1$.

When $f\lesssim_{\nu_1, \ldots, \nu_r}1$, we will say that $f$ is $(\nu_1, \ldots, \nu_r)$-bounded.

\subsection{Curves}\label{sect:curves}

A \emph{curve} $\Gamma\subset\reals^d$ (without further explicit or implicit qualification) refers to a one-dimensional smooth embedded submanifold of $\reals^d$.  

For a field $\mathbb{F}$ and an ideal $I\subset \mathbb{F}[X_1,\ldots,X_d]$ of a polynomial ring in $d$ variables over $\mathbb{F}$, we define the (affine) zero-set \[Z_\mathbb{F}(I)=\{x\in\mathbb{F}^d\,\,|\,\,f(x)=0\text{ for all }f\in I\}.\] We will mostly be interested in zero-sets for the case where $\mathbb{F}=\complex$; see \cite{ideals}, \cite{shafarevich} and \cite{hartshorne} for an introduction to algebraic geometry in this setting. In particular, we will assume that the reader is familiar with basic notions such as irreducibility, the dimension of ideals and singularities of zero sets but we will briefly review concepts and results which are more advanced.

\begin{definition}[Algebraic curve]
An \emph{(affine) algebraic curve in ambient dimension} $d$, $\Gamma$, is the zero set in $\complex^d$ of a one-dimensional ideal in $\complex[X_1, \ldots, X_d]$.

The \emph{ideal of }$\Gamma$ is the ideal \[I_\Gamma := \{f\in\complex[X_1,\ldots,X_d]\,\,|\,\,f(x)=0\text{ for all }x\in \Gamma\}.\]
\end{definition}

For certain technical reasons which can arise in dimensions $d>2$, we will restrict to real curves whose complexification is one-dimensional according to the following definition.

\begin{definition}[Real algebraic curve]
A \emph{real algebraic curve} $\Gamma$ in $\reals^d$ is the non-empty open subset in $\reals^d$ of a set of the form $Z_\complex(I)\cap\reals^d$ such that 

\begin{enumerate}
\item The ideal $I\subset\complex[X_1,\ldots,X_d]$ is one-dimensional (over $\complex$).
\item For each irreducible component $C$ of $Z_\complex(I)$, the set $C\cap\reals^d$ is a one-dimensional smooth embedded one-dimensional submanifold of $\reals^d$ away from the singularities of $Z_\complex(I)$.
\end{enumerate}

\end{definition}

\begin{remark} With this definition, for example, although the zero set of $f(X_1, X_2, X_3)=X_2^2+X_3^2$ in $\reals^3$ is a one-dimensional smooth manifold (it is the line along the $X_3$-axis), it is not a real algebraic curve since the complexification has dimension $2$ over $\complex^3$. On the other hand, even though the zero set of $f(X_1, X_2)=X_1^2+X_2^2$ in $\complex^2$ is one-dimensional, it is not a real algebraic curve since its intersection with $\reals^2$ is a single point.
\end{remark}

We will frequently consider smooth parametrizations of subsets of real algebraic curves so that we can apply analytic tools in our arguments. When there is a designated smooth parametrization $\gamma$ of $\Gamma$, we will sometimes abuse our definition slightly and refer to the parametrization $\gamma$ itself as $\Gamma$; this will be clear from context.

\begin{definition}[Degree] Let $\Gamma$ be a real algebraic curve. The \emph{(geometric) degree of} $\Gamma$ is the geometric degree of $Z_\complex(I_\Gamma)$, i.e. the number of points of intersection of the projective closure of $Z_\complex(I_\Gamma)$ with a generic hyperplane. The  \emph{ambient dimension} of $\Gamma$ is the complex ambient dimension of $Z_\complex(I_\Gamma)$.
\end{definition}

\begin{definition}[Algebraic degree]
Let $Z_\complex(I)\subset\complex^d$ be the zero-set of an ideal $I\subset\complex[X_1, \ldots, X_d]$. The set $\{f_1, \ldots, f_r\}\subset\complex[X_1, \ldots, X_d]$ \emph{generates the zero-set} $Z_\complex(I)$ if the ideal $J$ generated by $\{f_1, \ldots, f_r\}$ has $Z_\complex(J)=Z_\complex(I)$ (equivalently, if the radical ideals generated by $I$ and $J$ coincide). For a real algebraic curve $\Gamma\subset\reals^d$, the \emph{algebraic degree of} $\Gamma$ is the minimum of $\max_j\deg f_j$ over all sets $\{f_1, \ldots, f_r\}$ generating $\Gamma$.
\end{definition}

\begin{remark}
In the case where the ambient dimension is $d=2$, the algebraic and geometric degrees of an irreducible real algebraic curve coincide.
\end{remark}

\subsection{Computational algebraic geometry}

We will now review some quantitative tools from algebraic geometry which will be useful in deriving bounds for quantities arising in our proof.

We will utilize a refinement of B\'{e}zout's Theorem and two standard corollaries (proved here for completeness) to bound the number of zero-dimensional components in intersections. The proof of this refinement appears in \cite{heintz}. For a more detailed exposition of this result, see the section on B\'{e}zout's inequality in \cite{spendsymm}.

\begin{theorem}[B\'{e}zout]
Let $I\subset\complex[X_1,\ldots,X_d]$ be an ideal generated by $\{f_1,\ldots,f_r\}$ with $r\ge d$. Assume that $\deg f_j\ge\deg f_{j+1}$ for $j=1, \ldots, (r-1)$. Then the number of zero-dimensional components of $Z_\complex(I)$ is at most \[\prod_{j=1}^d(\deg f_j).\]
\end{theorem}

\begin{remark}
Note that only the $d$ largest degrees appear in the product for the upper bound of the number of zero-dimensional components.
\end{remark}

\begin{corollary}
Let $\Gamma\subset\reals^d$ be an irreducible real algebraic curve of algebraic degree $m$. Then the number of singularities of $\Gamma$ is $(d,m)$-bounded.
\end{corollary}

\begin{proof}
Let $\Gamma$ be generated by $\{f_1, \ldots, f_r\}$ where $\deg f_j\le m$ for each $1\le j\le r$. The singularities of $\Gamma$ form a Zariski-closed proper subset of $\Gamma\subset\complex^d$ which is the intersection of $\Gamma$ with hypersurfaces which are the zero-sets of determinants of the $(d-1)\times (d-1)$ minors of the Jacobian matrix of $\{f_1, \ldots, f_r\}$. These determinants are polynomials of $(d,m)$-bounded degree. By the irreducibility of $\Gamma$, at least one of these hypersurfaces intersects $\Gamma$ in a finite number of points. By B\'{e}zout's Theorem, the number of singularities is therefore $(d,m)$-bounded.
\end{proof}

\begin{corollary}
Let $\Gamma\subset\reals^d$ be an irreducible real algebraic curve of geometric degree $n$ and algebraic degree $m$. Then $n\le m^d$.
\end{corollary}

\begin{proof}
By B\'{e}zout's Theorem, the number of zero-dimensional components in the intersection of $\Gamma$ with the zero set of a linear polynomial is $\le m^d$.
\end{proof}

When the ambient dimension is $d>2$, we will use the theory of Gr\"{o}bner bases to obtain appropriate bounds. For an introduction to Gr\"{o}bner bases, see \cite{ideals}. The degree bound we will use is the following result due to Dub\'{e}.

\begin{theorem}[Dub\'{e} \cite{dubebound}]
Suppose that $I\subset\complex[X_1,\ldots,X_d]$ is an ideal generated by $\{f_1, \ldots, f_r\}$. Write $m=\max_j\deg f_j$. Then there exists a Gr\"{o}bner basis of $I$ consisting of polynomials all of degree $\lesssim m^{2^d}$.
\end{theorem}

If $I\subset\complex[X_1, \ldots, X_{d_1}, Y_1, \ldots, Y_{d_2}]$ is an ideal such that $\dim Z_\complex(I)=r$, then the Zariski closure of the projection of $Z_\complex(I)$ onto the first $d_1$ coordinates has dimension at most $r$. The ideal corresponding to this Zariski closure is precisely the ideal $I\cap\complex[X_1, \ldots, X_{d_1}]$, up to taking radicals. Recall that if $G$ is a Gr\"{o}bner basis of $I$ for an elimination ordering eliminating $Y_1, \ldots, Y_{d_1}$, then $G\cap\complex[X_1, \ldots, X_{d_1}]$ is a Gr\"{o}bner basis for $I\cap\complex[X_1, \ldots, X_{d_1}]$.

We will also require a result which bounds the number of connected components of the real algebraic curve $\Gamma\subset\reals^d$. When $d=2$, we may use Harnack's Curve Theorem \cite{harnack}. For $d>2$, there is the following deeper result due to Thom \cite{thom} and Milnor \cite{milnor}; we will only state the theorem for the situation which arises in this article.

\begin{theorem}[Thom-Milnor]
Let $\Gamma\subset\reals^d$ be a real algebraic curve of algebraic degree $m$ and ambient dimension $d$. The number of connected components of $\Gamma$ in $\reals^d$ is at most $m^{Cd}$ for some universal constant $C>0$.
\end{theorem}

\begin{remark}\label{remark:difftoalg}
At various points in our proof, we will need to convert ordinary differential equations to a more algebraic form.

Suppose that the set $\{f_1, \dots, f_r\}$ generates the real-algebraic curve $\Gamma$ with $\deg f_j\le m$ and let $\gamma:I\to\reals^d$ be a singularity-free real-analytic parametrization of an open subset of $\Gamma$. Write $\gamma(\tau)=(x_1(\tau), \ldots, x_d(\tau))$ for real-analytic $x_i:I\to\reals$ and $\dot{\bf x}=(\dot x_1, \ldots, \dot x_d)$. Since $f_j(\gamma(\tau))\equiv 0$, it follows that \[\nabla f_j(x_1,\ldots,x_d)\cdot(\dot x_1, \ldots, \dot x_d)\equiv 0.\] For each $1\le i\le d$, write $\dot {\bf x}^i\in\reals^{d-1}$ for the vector obtained by deleting $\dot x_i$ from $\dot {\bf x}$. For each $(d-1)\times (d-1)$ minor $M$ of the Jacobian $J$ associated to $\{f_1, \ldots, f_r\}$ whose columns do not include the $i$th column of $J$, we get an equation \[M\dot{\bf x}^i=\dot x_i{\bf b},\] for a certain vector ${\bf b}$ of partial derivatives with respect to $x_i$ of $\{f_1, \ldots, f_r\}$. Hence, \begin{equation}\label{tangentexpress}\det(M)\dot{\bf x}^i=\dot x_iM^*{\bf b}\end{equation} where $M^*$ is the adjugate matrix of $M$. Note that $\det M$ and the entries in $M^*$ are polynomials in $x_1, \ldots, x_d$ of degree at most $dm$. Furthermore, not all of these equations can be trivial on $\Gamma$ since, away from a finite number of singularities, the rank of $J$ is $(d-1)$. 

Consequently, any first order differential equation of degree $\kappa$ satisfied by the components of $\gamma$ on an open set is equivalent to the vanishing on $\Gamma$ of a certain system of polynomials of $(d,m,\kappa)$-bounded degree and, conversely, if these polynomials do not all vanish on $\Gamma$ then the differential equation fails to be satisfied for some singularity-free open subset of $\Gamma$.

Similarly, using the equation \[0\equiv\frac{d}{dt}\big(\nabla f_j({\bf x})\cdot\dot{\bf x})=\dot{\bf x}\cdot H^{f_j}_{\bf x}(\dot{\bf x}) + \nabla f_j({\bf x})\cdot\ddot{\bf x},\] where $H^f_{\bf x}$ is the Hessian of $f$ at ${\bf x}$, allows us to express any second-order differential equation of degree $\kappa$ in the components of $\gamma$ as a system of polynomials of $(d,m,\kappa)$-bounded degree such that the differential equation is satisfied by $\Gamma$ if and only if all the polynomials vanish on $\Gamma$.
\end{remark}

We will treat the case of rationally-parametrized curves first because this case is more elementary and we can (often) obtain a better bound. For an introduction to rational curves, see \cite{rationalintro}.

\begin{definition}
A smooth function $\gamma:I\to\reals^d$ where $I$ is a real open interval is a \emph{(real) rational parametrization} if each coordinate function $\gamma_j(t)$ for $j=1, \ldots, d$ is given by a reduced rational function (in other words, the ratio of two coprime polynomials) in $t$ over $\reals$ and the tangent vector $\dot\gamma(t)$ does not vanish for $t\in I$. The \emph{degree of} $\gamma$ is $\max_j\deg(\gamma_j)$.
\end{definition}

Such functions $\gamma$ are parametrizations of one-dimensional open subsets of the intersection of complex curves with $\reals^d$ of bounded algebraic degree. Indeed, let $\gamma:I\to\reals^d$ be a rational parametrization of degree $m$ and write $\gamma_j(t)=f_j(t)/g_j(t)$ for coprime polynomials $f_j$, $g_j$ of degree bounded by $m$. By considering a Gr\"{o}bner basis for an elimination ordering for $t$ of the ideal generated by $\{g_jX_j-f_j\,\,|\,\, 1\le j\le d\}$ in $\complex[X_1, \ldots, X_d, t]$ and eliminating $t$, it follows that such a rational parametrization $\gamma$  defines the complex parametrization of an open subset of an algebraic curve in $\complex^d$. By analytic continuation, $\gamma$ must parametrize an open subset of an irreducible component of this curve. Invoking Dub\'{e}'s bound, the algebraic degree of this irreducible curve over $\complex$ is at most $\lesssim m^{2^d}$.

\begin{remark}\label{param}
In the special case where the ambient dimension $d=2$, we can obtain a better bound for the degree of the implicit algebraic equation by considering instead the resultant eliminating $t$ of $\{g_1X_1-f_1,g_2X_2-f_2\}$. This is a polynomial of degree at most $m$ in $\complex[X_1,X_2]$, so $\gamma$ parametrizes the open subset of the intersection with $\reals^2$ of an irreducible algebraic curve of (geometric or algebraic) degree at most $m$.
\end{remark}

\subsection{Combinatorial geometry}

We will require a variation of the Szemer\'{e}di-Trotter Theorem (which bounds the number of incidences between a set of points in the plane and a set of lines in terms of the number of points and lines) in our proof.

\begin{definition}[Admissible]
Let ${\bf \Gamma}$ be a finite collection of curves in $\reals^2$ and $C$ be a positive integer. The collection ${\bf\Gamma}$ is $C$\emph{-admissible}, if the following two conditions hold:
\begin{enumerate}
\item Any two distinct curves $\Gamma_1,\Gamma_2\in{\bf\Gamma}$ meet in at most $C$ points of $\reals^2$.
\item Any two distinct points in $\reals^2$ are incident to at most $C$ curves from ${\bf\Gamma}$.  
\end{enumerate}
\end{definition}

We will use the following variant due to Pach and Sharir \cite{szemtrot}.

\begin{theorem}[Pach-Sharir]\label{theorem:pachsharir}
Let ${\bf \Gamma}$ be a finite collection of curves and $Q$ be a finite collection of points in $\reals^2$. If ${\bf \Gamma}$ is $C$-admissible and each curve $\Gamma\in{\bf\Gamma}$ does not intersect itself, then the number of incidences, $I({\bf\Gamma},Q)$, between ${\bf \Gamma}$ and $Q$ satisfies,
\[I({\bf\Gamma},Q)=|\{(\Gamma,q)\in{\bf\Gamma}\times Q\,\,|\,\, q\in\Gamma\}|\lesssim_{C}|{\bf\Gamma}|^{\frac{2}{3}}|Q|^{\frac{2}{3}}+|{\bf\Gamma}|+|Q|.\]
\end{theorem}

\subsection{Structural rigidity}

We will introduce some definitions from the theory of structural rigidity; we have adapted them from the standard ones to be more suited to our particular application of algebraic curves embedded in an ambient Euclidean space. The reader may consult \cite{rigid1} and \cite{rigid2} for some background, although we will not assume the reader has knowledge of this area and we define the terminology used in the paper below.

\begin{definition}[Framework]
Let $\mathcal{G}=\mathcal{G}(V,E)$ be a graph with vertex set $V$ and edge set $E$. Let $M$ be a subset of $\reals^d$. A $\mathcal{G}$\emph{-framework on }$M$ is a drawing of $\mathcal{G}$ in $\reals^d$ such that all vertices are distinct and lie on $M$.

If $\phi:V\to M$ is an injective map, the $\mathcal{G}$-framework on $M$ with each vertex $v\in V$ corresponding to $\phi(v)\in M$ will be denoted by $\mathcal{G}^M(\phi)$. 
\end{definition}

Fix an ambient dimension $d>1$ and a smooth function $D:\reals^{d}\times\reals^d\to\reals$ which we write as $D(x,y)$ for $x,y\in\reals^d$.

\begin{definition}[Flexible framework]
Let $\mathcal{G}(V,E)$ be a graph and $M\subset\reals^d$. The framework $\mathcal{G}^M(\phi)$ is $D$\emph{-flexible on }$M$ if there exists a continuous function $\Phi:V\times (-\delta,\delta)\to M$ for some $\delta>0$ such that, writing $\phi_t(v):=\Phi(v,t)$, it is true that $\phi_0= \phi$, there exists $t_0\in(-\delta,\delta)$ such that $\phi_{t_0}\ne\phi$ and, for each pair of edges $v,w\in E$, the edge function \[t\mapsto D(\phi_t(v),\phi_t(w))\] is constant.

We will say that $\mathcal{G}^M(\phi)$ is $D$\emph{-smoothly flexible on }$M$ if, for each $v\in V$, the map $t\mapsto \phi_t(v)$ is smooth.

The function $\Phi$ is a $D$-\emph{motion} of $\mathcal{G}^M(\phi)$.
\end{definition}

Informally, a $D$-flexible framework on $M$ is an vertex embedding of a graph into $M$ which can be moved continuously while preserving the value of $D$ along each edge of the graph; the $D$-motion is the function which describes this movement.

\begin{remark}
In the structural rigidity literature, it is common to ignore motions arising from symmetries of $D$ (e.g. rigid motions when $D$ is the square-distance function); for convenience, we will not do follow this convention.
\end{remark}

We will also be interested in infinitesimal $D$-motions.

\begin{definition}[Infinitesimally flexible framework]
Let $\mathcal{G}(V,E)$ be a graph and $M\subset\reals^d$ be a smooth embedded submanifold of $\reals^d$. The framework $\mathcal{G}^M(\phi)$ is $D$\emph{-infinitesimally flexible on }$M$ if for each $v\in V$ there exists a tangent vector $t(v)\in T_{\phi(v)}M\subset\reals^d$ such that for each pair of vertices $v,w\in V$, \[\nabla D(v,w)\cdot(t(v),t(w))=0.\]
\end{definition}

Informally, an infinitesimal motion is an assignment of velocity vectors to each embedded vertex in such a way that the value of $D$ along each edge remains constant up to first order.

\begin{remark}
By considering the derivative at $t=0$ of the edge function $t\mapsto D(\phi_t(v),\phi_t(w))$ in the definition of $D$-flexibility, it follows that if a framework $\mathcal{G}^M(\phi)$ on a smooth embedded submanifold $M\subset\reals^d$ is $D$-smoothly flexible then it is $D$-infinitesimally flexible.
\end{remark}

The \emph{bipartite graph }$\mathcal{K}_{m,n}$ is the graph with vertex set $V_1\cup V_2$ for disjoint sets $V_1$ and $V_2$ with cardinalities $|V_1|=m$, $|V_2|=n$ and edge set $\{uv\,\,|\,\,u\in V_1, v\in V_2\}$. 

The framework $\mathcal{K}_{m,n}^M(\phi)$ will be written as \[\mathcal{K}_M(\phi(V_1),\phi(V_2))\] and referred to as an $(m,n)$\emph{-framework on }$M$. Note that this is well-defined up to permutation of each vertex set $V_1, V_2$.

The \emph{complete graph} $\mathcal{K}_N$ is the graph with vertex set $V$ such that $|V|=N$ and edge set $\{uv\,\,|\,\,u,v\in V\}$. The \emph{triangular graph} $\mathcal{T}$ is the complete graph $\mathcal{K}_3$ on a set of three vertices $V$. The framework $\mathcal{T}^M(\phi)$ will also be written as \[\mathcal{T}_M(\phi(V)).\] This is well-defined up to permuting $V$. We will say that $\mathcal{T}^M(\phi)$ is \emph{based at }$\{x,y\}$ for $x,y\in M$ distinct if $\{x, y\}\subset\phi(V)$.

In the context of $D$-flexibility along curves, we will be interested in the following degeneracy condition.

\begin{definition}[Degenerate curve]\label{def:degen}
Let $\mathcal{G}$ be a graph. A smooth embedded curve $\Gamma\subset\reals^d$ is $(D,\mathcal{G})$\emph{-degenerate} if every $\mathcal{G}$-framework on $\Gamma$ is $D$-smoothly flexible. 
\end{definition}

Informally, $\Gamma$ is $(D,\mathcal{G})$-degenerate if every vertex embedding of $\mathcal{G}$ into $\Gamma$ can be moved smoothly along $\Gamma$ while preserving the value of $D$ along each edge.

\begin{remark}
When $D$ is clear from context and especially when we are considering the square distance function \[D(X,Y)=\norm{X-Y}^2:=\norm{X-Y}_{\reals^d}^2,\] we will often suppress reference to $D$ in the definitions above.
\end{remark}

\section{Step 1: Reduction to a two-dimensional problem}\label{sect:gen}

Let $D:\reals^d\times\reals^d\to \reals$ be a real polynomial in $2d$ variables. We will primarily be interested in the case where $D$ is the square distance function, $D(X,Y)=\norm{X-Y}^2$, but we will also consider more general $D$.

\begin{definition}[Distance polynomial]\label{def:distpoly}
Let $D:\reals^d\times\reals^d\to\reals$ be a real polynomial in $2d$ variables and $\gamma:I\to\reals^d$ be a smooth function. Then $D$ is a \emph{distance polynomial for} $\gamma$ if the following conditions hold:
\begin{enumerate}
\item $D(\gamma(\alpha),\gamma(\beta))=D(\gamma(\beta),\gamma(\alpha))$ for all $\alpha,\beta\in I$.
\item $D(\gamma(\alpha),\gamma(\beta))=0$ if and only if $\alpha=\beta$.
\end{enumerate}
\end{definition}

\begin{remark}
The square distance function $D(X,Y)=\norm{X-Y}^2$ is a distance polynomial for any injective $\gamma:I\to\reals^d$.
\end{remark}

Let $I$ be a non-empty open interval in $\reals$ and let $\gamma:I\to\reals^d$ be an injective real-analytic parametrization of a curve $\Gamma$ in $\reals^d$. Let $D$ be a distance polynomial for $\gamma$.  Let $P\subset \Gamma$ be a finite set of points lying on the curve. Write \[\Delta_D(P)=\{D(p,q)\,\,|\,\,p,q\in P\}\] for the image of $P\times P$ under $D$.

For each pair of points $(p,q) \in P^{2*}$, consider the smooth map $\xi_{pq}:I\to\reals^2$ given by \[\xi_{pq}(t)=(D(\gamma(t),p),D(\gamma(t),q)).\]

\begin{definition}[Elekes curves]\label{def:elekescurves}
The \emph{Elekes curve} $\Xi_{pq}$ is the curve in $\reals^2$ with smooth parametrization $\xi_{pq}$. The \emph{set of Elekes curves} corresponding to $P$ is the set \[{\bf\Xi}_P=\{\Xi_{pq}\,\,|\,\,(p, q)\in P^{2*}\}.\]
\end{definition}

A key observation is the following exponent gap result for the cardinality of $|\Delta(P)|$. This follows from a method of Elekes who used it to derive a quantitative bound \cite{elekes} in his proof of Purdy's Conjecture; the original proof of the conjecture (without a quantitative bound) is from Elekes-R\'{o}nyai \cite{elekesronyai}.

\begin{proposition}\label{prop:xixi}
Suppose that there exists a subset ${\bf \Xi}'\subset{\bf\Xi}_P$ which is $C$-admissible and consists only of curves which do not intersect themselves. If $|{\bf\Xi'}|\ge c_0|P|^2$, then \[|\Delta_D(P)|\gtrsim_{C, c_0} |P|^{1+\frac{1}{4}}.\] 
\end{proposition}

\begin{proof}
Each curve $\Xi_{pq}\in{\bf \Xi'}$ does not intersect itself, so it is incident to $|P|-2$ distinct points of the Cartesian product $\Delta_D(P)^2\subset\reals^2$, namely the set of points \[\{(D(r,p),D(r,q))\,\,|\,\,r\in P\setminus\{p,q\}\}.\] Therefore, the number of incidences $I({\bf\Xi'},\Delta_D(P)^2)$ satisfies \[I({\bf\Xi'},\Delta_D(P)^2)\ge(|P|-2)|{\bf\Xi'}|\gtrsim_{c_0}|P|^3.\]

Theorem~\ref{theorem:pachsharir} applied to ${\bf\Xi'}$ and $\Delta_D(P)^2$ gives \[I({\bf\Xi'},\Delta_D(P)^2)\lesssim_{C}|{\bf\Xi'}|^{\frac{2}{3}}|\Delta_D(P)|^{\frac{4}{3}}+|{\bf\Xi'}|+|\Delta_D(P)|^2.\] 

Combining the two bounds for the number of incidences and using the trivial bound $|{\bf\Xi'}|\le|P|^2$ yields
\[|P|^3\lesssim_{C,c_0}|P|^{\frac{4}{3}}|\Delta_D(P)|^{\frac{4}{3}}+|P|^2+|\Delta_D(P)|^2\]
which gives the stated lower bound on $\Delta_D(P)$.
\end{proof}

The strategy for proving lower bounds for $|\Delta_D(P)|$ will thus be to show that ${\bf\Xi}_P$ contains many distinct curves (i.e. $\gtrsim |P|^2$) and that a positive proportion of the set of distinct curves form an admissible set in the above sense.  

It will be convenient to reduce matters to curves which are well-behaved in the following sense.

\begin{definition}[Simple pair]\label{def:simple}
Let $D:\reals^{2d}\to\reals$ be a polynomial and let $\Gamma$ be a curve which has a real-analytic parametrization $\gamma:I\to\reals^d$ for some non-empty open interval $I\subset \reals$. The pair $[D,\Gamma]$ is \emph{simple} if the following conditions hold:
\begin{enumerate}
\item The parametrization $\gamma$ is injective and singularity-free.
\item The function $t\mapsto\ddot\gamma(t)$ does not vanish identically on $I$.
\item The polynomial $D$ is a distance polynomial for $\gamma$.
\item For each $\alpha, \beta\in I$, the map $I\to\reals^2$ given by \[t\mapsto (D(\gamma(t),\gamma(\alpha)),D(\gamma(t),\gamma(\beta)))\] is injective.
\item The map $\{(\alpha,\beta)\in I^2\,\,|\,\,\alpha\ne\beta\}\to\reals$ given by $(\alpha,\beta)\mapsto D(\gamma(\alpha),\gamma(\beta))$ is a submersion (i.e. its differential does not vanish).
\end{enumerate}
\end{definition}

\begin{remark}
For specific choices of $D$ and $\Gamma$, it is a technical matter to determine whether $[D,\Gamma]$ is simple; one general strategy which appears to work widely is to split $\Gamma$ into a controlled number of pieces and deal with each separately (see, for example, Section~\ref{sect:simplicity}). While the conditions above are chosen to be general enough to include the cases of principal interest in this paper but specific enough to make the subsequent arguments as elementary as possible, we do not believe that this class of $[D,\Gamma]$ is in any sense optimal or the most natural if one seeks to make a fully general statement analogous to Theorem~\ref{thm:algdist}.
\end{remark}

We will frequently make use of the following almost immediate consequence of the definition; informally, it states that any line segment or V-shaped graph (i.e. the bipartite graphs $\mathcal{K}_{1,1}$ and $\mathcal{K}_{2,1}$) with vertices on the curve can be moved along the curve while preserving the values of $D$ along the edges.

\begin{lemma}\label{lemma:difftopert}
Let $[D,\Gamma]$ be simple. Then $\Gamma$ is $(D,\mathcal{K}_{1,1})$-degenerate and $(D,\mathcal{K}_{2,1})$-degenerate.
\end{lemma}

\begin{proof}
Suppose that $\alpha_0, \beta_0 \in I$ are distinct and \[D(\gamma(\alpha_0), \gamma(\beta_0)) = d.\] By condition $5$ in the definition of a simple pair, the Implicit Function Theorem applies for the implicit equation \[D(\gamma(\alpha),\gamma(\beta))=d, (\alpha,\beta)\in I^2\] and we deduce that there exist small open neighbourhoods $U, V\subset I$ of $\alpha_0,\beta_0$ respectively and a smooth bijection $\beta:U\to V$ such that $\beta(\alpha_0)=\beta_0$ and \[D(\gamma(\alpha),\gamma(\beta(\alpha)))=d\text{ for all }\alpha\in U.\]

This implies that $\Gamma$ is $(D,\mathcal{K}_{1,1})$-degenerate. Repeating the argument (and replacing $U$ and $V$ with smaller neighborhoods of $\alpha_0, \beta_0$ if necessary) implies that it is $(D,\mathcal{K}_{2,1})$-degenerate.
\end{proof}

\begin{remark}\label{rem:uniqueline}
It should be observed that the $\beta$ constructed in the proof above is uniquely determined (locally) given the requirements that $\beta(\alpha_0)=\beta_0$ and \[D(\gamma(\alpha),\gamma(\beta(\alpha)))=d.\] Informally, this means that for simple $[D,\Gamma]$, given two points $p,q\in\Gamma$, when we move $p$ slightly along $\Gamma$ there is exactly one way to move $q$ along $\Gamma$ in such a way that the value of $D(p,q)$ is preserved throughout the motion.
\end{remark}

Our main general result which links algebra to rigidity along curves is the following. It states that the curve $\Gamma$ enjoys an exponent gap as in the statement of Theorem~\ref{thm:algdist} unless every triangle with vertices on the curve may be moved along the curve while preserving the values of $D$ along edges. Recall that $\mathcal{T}$ denotes the triangular graph.

\begin{theorem}\label{thm:general}
Suppose that $\Gamma\subset\reals^d$ is the singularity-free subset of a real algebraic curve of algebraic degree $m$ and $[D,\Gamma]$ is simple.

If $\Gamma$ is not $(D,\mathcal{T})$-degenerate then whenever $P\subset\Gamma$ is a finite subset, \[|\Delta_D(P)|\gtrsim_{m,d,\deg D}|P|^{1+\frac{1}{4}}.\]

Furthermore, if $\Gamma$ is rationally parametrized by $\gamma:I\to\reals^d$ then the implicit constant can be chosen to depend only on $\deg\gamma$ and $\deg D$.
\end{theorem}

\section{Step 2: Checking admissibility}\label{sect:admiss}

\subsection{Exponent gap for rational curves}\label{sect:ratcurves}

In this section, we will deal with the conceptually easier case where the curve $\Gamma$ is a rational curve. To this end, assume that $\Gamma$ is parametrized by $\gamma:I\to\reals^d$ where $I$ is an open interval and the components of $\gamma(t)$ are real rational functions of $t$ not all of which are constant. Let $D:\reals^d\times\reals^d\to\reals$ be a distance polynomial for $\gamma$.

With this setup, the components of each Elekes curve parametrization $\xi_{pq}(t)$ are rational functions of $t,\alpha=\gamma^{-1}(p),\beta=\gamma^{-1}(q)$. Under the additional assumption that $[D,\Gamma]$ is simple (recall Definition~\ref{def:simple}), the parametrization $\xi_{pq}$ is not constant (by condition $4$) and each curve $\Xi_{pq}$ is now a rational plane curve which is irreducible of degree $\lesssim\deg D\deg\gamma$ (see Remark~\ref{param}).  We will reduce to this case where $[D,\Gamma]$ is simple in Section~\ref{sect:simplerat}.

By B\'{e}zout's Theorem, any two curves $\Xi_{pq}$, $\Xi_{p'q'}$ intersect in fewer than $(\deg D\deg\gamma)^2$ points, unless both curves correspond to the same algebraic curve and have a non-empty open subset in common. 

Write \[\xi_{pq}(t)=\left(\frac{f_1(t)}{g_1(t)}, \frac{f_2(t)}{g_2(t)}\right)\] for polynomials $f_1(t), f_2(t), g_1(t), g_2(t)$ in $t$ of degree $\lesssim\deg D\deg\gamma$ with coefficients which are polynomials in $\alpha, \beta$ of degree $\lesssim\deg D\deg\gamma$. Define $G_{\alpha,\beta}(X,Y)$ to be the resultant eliminating $t$ of $\{g_1X-f_1,g_2Y-f_2\}$. Then $G_{\alpha,\beta}$ has degree $\lesssim\deg D\deg\gamma$ and its coefficients are polynomials in $\alpha,\beta$ of $(\deg D\deg\gamma)$-bounded degree. 

The following lemma provides the aforementioned link to rigidity along the curve. It implies that if there are many incidences between Elekes curves then there is a flexible $(2,k)$-framework on $\Gamma$.

\begin{lemma}\label{lemma:rigid}
There exists a positive integer $\mu$ sufficiently large depending on $\deg\gamma$ and $\deg D$ with the following property: Let $k\ge 2$ and suppose that $\delta_1, \ldots, \delta_k\in\reals^2$ are distinct points. Consider $\mu$ distinct pairs $(\gamma(\alpha_1),\gamma(\beta_1)), \ldots, (\gamma(\alpha_\mu),\gamma(\beta_\mu))\in P^{2*}$ such that all $\mu$ curves $\Xi_{\gamma(\alpha_j)\gamma(\beta_j)}$ for $j=1,\ldots, \mu$ are incident to all $\delta_i$ for $1\le i\le k$. Then there exist $1\le j\le \mu$ and distinct $t_1, \ldots, t_k\in \Gamma$ such that the $(2,k)$-framework \[\mathcal{K}_\Gamma(\{\gamma(\alpha_j),\gamma(\beta_j)\},\{t_1, \ldots, t_k\})\] is $D$-smoothly flexible on $\Gamma$.
\end{lemma}

We will write $\mu(\deg\gamma,\deg D)$ for the smallest $\mu$ for which the conclusion holds.

\begin{proof}
For each point $\delta\in\reals^2$ we obtain a polynomial $H_{\delta}(X,Y):= G_{X,Y}(\delta)$ of $\deg D\deg\gamma$-bounded degree such that whenever $\xi_{\gamma(\alpha)\gamma(\beta)}(\tau)=\delta$, it follows that $H_{\delta}(\alpha,\beta)=0$. By dividing by suitable polynomial factors (depending on $\delta$) if necessary, we may assume without loss of generality that each $H_\delta$ is square-free. 

The $k$ polynomials $H_{\delta_j}$ have $\mu$ common zeroes, namely the set $\{(\alpha_j,\beta_j)\}_{j=1}^\mu$. By B\'{e}zout's Theorem, the number of zero-dimensional components of the ideal $J$ generated by the $H_{\delta_j}$ is $\deg D\deg\gamma$-bounded. Therefore, if we choose $\mu$ sufficiently large depending only on $\deg\gamma$ and $\deg D$, there is some $(\alpha_j, \beta_j)$ which lies on a one-dimensional component of $Z_\complex(J)$.   

For each $1\le i\le k$, there exists $\tau_i\in I$ such that 
\begin{equation}\label{disteqn}(D(\gamma(\tau_i),\gamma(\alpha_j)),D(\gamma(\tau_i),\gamma(\beta_j)))=\delta_i.\end{equation} By Lemma~\ref{lemma:difftopert}, for each $i$, we may perturb $\alpha_j$ and redefine $\tau_i$ and $\beta_j$ to vary smoothly with $\alpha_j$ while preserving (\ref{disteqn}). The point $(\alpha_j, \beta_j)$ therefore lies on a one-dimensional irreducible component of $Z_\complex(H_{\delta_i})$ whose intersection with $\reals^2$ is also one-dimensional and $(\alpha_j, \beta_j)$ may be perturbed along $Z_\complex(H_{\delta_i})\bigcap\reals^2$. Consequently, $(\alpha_j, \beta_j)$ may be perturbed along $Z_\complex(H_{\delta_i})\bigcap\reals^2$ simultaneously for all $1\le i\le k$.
\end{proof}

To prove Theorem~\ref{thm:general} (for rationally parametrized curves) we combine this lemma with two structural rigidity results about triangular frameworks on $\Gamma$. The first result essentially shows that there is a bounded $k$ so that if even one triangular framework on $\Gamma$ based at $p,q\in\Gamma$ is not flexible then any $(2,k)$-framework with $\{p,q\}$ as one of its vertex sets is not flexible. The second result states that if $\Gamma$ is not $\mathcal{T}$-degenerate then, after ignoring a small number of points of $P$, we may assume that for every pair of points $p,q\in P$, there is a triangular framework based at $p,q$ which is not flexible. The proofs of both propositions are deferred to Section~\ref{sect:rig}.

\begin{proposition}\label{prop:rigiddist} Let $(p,q)\in P^{2*}$ and suppose that there exists a triangular framework based at $\{p,q\}$ which is not infinitesimally flexible on $\Gamma$. 

Then, $I$ may be partitioned into a $(\deg\gamma, \deg D)$-bounded number of intervals with non-empty interiors \[I=\bigcup_r I_r\] such that whenever $\tau_1, \tau_2\in I$ are distinct points and the $(2,2)$-framework \[\mathcal{K}_\Gamma(\{p,q\},\{\gamma(\tau_1),\gamma(\tau_2)\})\] is infinitesimally flexible along $\Gamma$, the points $\tau_1$ and $\tau_2$ do not lie in the same $I_r$.

In particular, for $k$ sufficiently large depending only on $\deg\gamma$ and $\deg D$ and any distinct points $t_1, \ldots, t_k\in\Gamma$, the $(2,k)$-framework \[\mathcal{K}_\Gamma(\{p,q\},\{t_1, \ldots, t_k\})\] is not infinitesimally flexible on $\Gamma$.
\end{proposition}

\begin{proposition}\label{prop:rigidzar} 
Suppose that $\Gamma$ is not $(D,\mathcal{T})$-degenerate. Then there exists $P_0\subset P$ such that $|P_0|\gtrsim_{\deg\gamma, \deg D}|P|$ and for each pair $(p,q)\in P_0^{2*}$ there exists a triangular framework based at $\{p,q\}$ which is not $D$-infinitesimally flexible on $\Gamma$. 
\end{proposition}

Armed with Lemma~\ref{lemma:rigid} and the two propositions, we can now prove Theorem~\ref{thm:general}.

\begin{proof}[Proof of Theorem~\ref{thm:general} for rationally parametrized curves]
Suppose that the curve $\Gamma$ is not $(D,\mathcal{T})$-degenerate. We replace $P$ with the set $P_0$  in the conclusion of Proposition~\ref{prop:rigidzar}  and lose a $(\deg\gamma, \deg D)$-bounded constant factor. 

With this reduction, one corollary of Proposition~\ref{prop:rigiddist} is that the Elekes curves, ${\bf\Xi}_P$, define many distinct algebraic curves. Indeed, suppose that the curves $\Xi_{p_1q_1},\ldots, \Xi_{p_rq_r}$ for distinct pairs $(p_1,q_1),\ldots, (p_r,q_r)\in P^{2*}$ all determine the same algebraic curve given by the irreducible polynomial $G(X,Y)$ of degree $\lesssim\deg D\deg\gamma$. Note that these curves may potentially be different (even potentially disjoint) as \emph{real} algebraic curves; recall the definitions from Section~\ref{sect:curves}.

The first coordinate of $\xi_{p_jq_j}(\gamma^{-1}(p_j))$ is $0$ and the first coordinate of $\xi_{p_jq_j}(t)$ for $t\ne\gamma^{-1}(p_j)$ is non-zero (by condition $2$ in Definition~\ref{def:distpoly} and the injectivity condition $4$ in Definition~\ref{def:simple}). There are only finitely many points with first coordinate $0$ lying on the zero set of $G$. Moreover, near the line $X=0$, the zero set of $G(X,Y)$ is the union of finitely many curves $\Theta_1,\ldots, \Theta_s$ whose number is $(\deg\gamma, \deg D)$-bounded and such that, for each $j=1,\ldots, r$, the curve $\Xi_{p_jq_j}$ contains one of the curves $\Theta_i$ entirely. By Proposition~\ref{prop:rigiddist}, there exists a positive integer $k=k(\deg\gamma, \deg D)$ depending only on $\deg\gamma$ and $\deg D$ such that for any $(p,q)\in P^{2*}$ and distinct $t_1, \ldots, t_k\in \Gamma$, the $(2,k)$-framework $\mathcal{K}((p,q),(t_1, \ldots, t_k))$ is not $D$-infinitesimally flexible along $\Gamma$. Pick $k(\deg\gamma, \deg D)$ points on each curve $\Theta_i$. By Lemma~\ref{lemma:rigid}, it follows that each $\Theta_i$ can be contained in at most $\mu(\deg\gamma, \deg D)$ of the curves $\Xi_{p_1q_1},\ldots, \Xi_{p_rq_r}$. Thus, $r\le s\mu(\deg\gamma, \deg D)$. Therefore the set ${\bf \Xi}_P$ contains a subset ${\bf\Xi'}$ consisting of $\gtrsim_{\deg\gamma,\deg D}|P|^2$ curves all of which determine different algebraic curves.

This set of curves ${\bf\Xi'}$ is not necessarily $K$-admissible for a $(\deg\gamma, \deg D)$-bounded $K$ since any two points of $\reals^2$ may potentially lie on several curves. To get around this, we replace each curve $\Xi_{pq}\in{\bf\Xi'}$ with the curve parametrized by the restriction $\xi_{pq}|_{I_k}$ to the interval $I_k\subset I$ from Proposition~\ref{prop:rigiddist} which contains the most points of $\Delta_D(P)\times\Delta_D(P)$; this number of points is at least $\gtrsim_{\deg\gamma, \deg D}|P|$ since the number of $I_k$ is $(\deg\gamma, \deg D)$-bounded. By Lemma~\ref{lemma:rigid} and Proposition~\ref{prop:rigiddist} it follows that this modified set ${\bf\Xi'}$ is $K$-admissible for a large enough $(\deg\gamma, \deg D)$-bounded $K$.

Applying Proposition~\ref{prop:xixi} completes the proof.

\end{proof}

\subsection{Exponent gap for algebraic curves}\label{sect:algcurves}

In this section, we prove Theorem~\ref{thm:algdist} for all algebraic curves; when the curve has a rational parametrization, the previous section usually gives better bounds.

Consider a real algebraic curve $\Gamma$ of algebraic degree $m$ with a real analytic parametrization $\gamma:I\to\reals^2$ such that $[D,\Gamma]$ is simple. We will reduce to this case in Section~\ref{sect:simplealg}. By considering the irreducible component of $\Gamma$ containing the most points of $P$ and losing a constant factor depending on $m$, we may assume that $\Gamma$ is irreducible. Let $\mathcal{I}_{X_1,\ldots,X_d}\subset\complex[X_1, \ldots, X_d]$ be a prime ideal generating $\Gamma$.

\begin{lemma}\label{lemma:2dcurve}
The parametrization $\xi_{pq}$ is injective and the Elekes curve $\Xi_{pq}$ is an open subset of an irreducible plane algebraic curve of $(d,m,\deg D)$-bounded degree. 
\end{lemma}

\begin{proof}
By condition $4$ in Definition~\ref{def:simple}, $\xi_{pq}$ is injective. Write $\gamma(t)=(x_1(t),\ldots,x_d(t))$ and $\xi_{pq}(t)=(A(t),B(t))$ for real-analytic functions $x_1,\ldots,x_d:I\to\reals$ and $A,B:I\to\reals$.  It follows that \begin{equation}\label{ABequ}A-D({\bf x},{\bf x}_p)=0\end{equation}\[B-D({\bf x},{\bf x}_q)=0\] where ${\bf x}_p$ and ${\bf x}_q$ are the coordinates of $p$ and $q$ respectively.

We work in the polynomial ring $R=\complex[A,B,{\bf x}_p,{\bf x}_q,{\bf x}]$ of $3d+2$ variables. By a slight abuse of notation, we will write $\mathcal{I}_{\bf x}$ for the ideal in $R$ given by substituting ${\bf x}$ for $(X_1,\ldots,X_d)$ in the ideal $\mathcal{I}_{X_1,\ldots,X_d}$. We will also define $\mathcal{I}_{{\bf x}_p}$ and $\mathcal{I}_{{\bf x}_q}$ similarly. Let $J\subset R$ be the ideal generated by the ideals $\mathcal{I}_{{\bf x}_p}$, $\mathcal{I}_{{\bf x}_q}$, $\mathcal{I}_{\bf x}$ and the two polynomials on the left-hand side of (\ref{ABequ}). Since $A, B$ are visibly uniquely determined given ${\bf x}$, ${\bf x}_p$ and ${\bf x}_q$ it follows that $\dim Z_\complex(J)=3$.

Consider the projection $\pi:\complex^{3d+2}\to\complex^{2d+2}$ onto the first $2d+2$ coordinates. The Zariski-closure of the projection $\pi(Z_\complex(J))$ has dimension at most $3$. The ideals $\mathcal{I}_{{\bf x}_p}$ and $\mathcal{I}_{{\bf x}_q}$ may be viewed as ideals in $\complex[A,B,{\bf x}_p,{\bf x}_q]$ and their zero sets then each have dimension $d+3$ in $\complex^{2d+2}$, since each is simply the Cartesian product of the irreducible algebraic curve $\Gamma$ with a Euclidean space of dimension $d+2$. Therefore, in $\complex^{2d+2}$, \[\dim Z_\complex(\mathcal{I}_{{\bf x}_p})\cap Z_\complex(\mathcal{I}_{{\bf x}_q})\ge 4.\] Thus the ideal \[J'=J\cap\complex[A,B,{\bf x}_p,{\bf x}_q]\] which corresponds to the Zariski closure of $\pi(Z_\complex(J))$ must contain polynomials which do not lie in the ideal $\mathcal{I}_{{\bf x}_p}+\mathcal{I}_{{\bf x}_q}\subset\complex[A,B,{\bf x}_p,{\bf x}_q]$.

Now, consider any ordering eliminating ${\bf x}$ in $R$. By Dub\'{e}'s bound, there is a Gr\"{o}bner basis for $J$ (with respect to this ordering) consisting of polynomials with $(d,m,\deg D)$-bounded degrees. Since $J'$ is non-zero and contains polynomials which do not lie in $\mathcal{I}_{{\bf x}_p}+\mathcal{I}_{{\bf x}_q}$, it follows that there exists a polynomial $G_{p,q}(A,B)\in\complex[A,B]$ of $(d,m,\deg D)$-bounded degree whose coefficients are polynomials in the coordinates of $p$ and $q$ of $(d,m,\deg D)$-bounded degree, not all of which vanish identically for $p$ or $q$ on $\Gamma$, such that whenever there exists $t\in I$ such that $\xi_{pq}(t)=(A,B)$, it follows that $G_{p,q}(A,B)=0$. By taking the real or imaginary part of $G_{p,q}$, we may assume that $G_{p,q}(A,B)\in\reals[A,B]$ (for $p,q\in\Gamma\subset\reals^d$). In particular, $\Xi_{pq}$ is the subset of the intersection with $\reals^2$ of an irreducible algebraic curve of $(d,m,\deg D)$-bounded degree; by condition $5$ in Definition~\ref{def:simple}, it is an open subset.
\end{proof}

By B\'{e}zout's Theorem, it follows that two curves $\Xi_{pq},\Xi_{p'q'}$ with $(p,q)\ne(p',q')$ are either defined by the same irreducible polynomial and they intersect in a non-empty open set or they meet in a $(d,m,\deg D)$-bounded number of points.

\begin{lemma}\label{lemma:rigidgen}
There exists a positive integer $\mu$ sufficiently large depending on $m$, $d$ and $\deg D$ with the following property: Let $k\ge 2$ and suppose that $\delta_1, \ldots, \delta_k\in\reals^2$ are distinct points. Consider $\mu$ distinct pairs $(p_1,q_1), \ldots, (p_\mu,q_\mu)\in P^{2*}$ such that all $\mu$ curves $\Xi_{p_jq_j}$ for $j=1,\ldots, \mu$ are incident to all points $\delta_i$ for $1\le i\le k$. Then there exist $1\le j\le \mu$ and distinct $t_1, \ldots, t_k\in \Gamma$ such that the $(2,k)$-framework \[\mathcal{K}_\Gamma(\{p_j,q_j\},\{t_1, \ldots, t_k\})\] is $D$-flexible on $\Gamma$.
\end{lemma}

\begin{proof}
For each $\delta\in\reals^2$, define the non-zero polynomial $H_\delta$ in the polynomial ring $\reals[{\bf x}_1,{\bf x}_2]$ of $2d$ variables by $H_\delta({\bf x}_1,{\bf x}_2):=G_{{\bf x}_1,{\bf x}_2}(\delta)$, where $G_{{\bf x}_1,{\bf x}_2}$ is the polynomial defined in the proof of Lemma~\ref{lemma:2dcurve}. Let $\mathcal{I}_\delta\subset\complex[{\bf x}_1,{\bf x}_2]$ be the ideal generated by $H_\delta$ and the ideals $\mathcal{I}_{{\bf x}_1}$, $\mathcal{I}_{{\bf x}_2}$. Then, whenever $\xi_{pq}(\tau)=\delta$ for some $\tau\in I$, the point $({\bf x}_p,{\bf x}_q)$ lies on a one-dimensional irreducible component of the zero set of $\mathcal{I}_\delta$.

Let $\mu$ be a positive integer. Consider the distinct points $\delta_1,\ldots, \delta_k\in\reals^2$ and the distinct pairs $(p_1,q_1), \ldots, (p_\mu,q_\mu)\in P^{2*}$. Suppose that for each $1\le i\le k$ and $1\le j\le \mu$, there exists $t_{ij}\in I$ such that $\xi_{p_j,q_j}(\tau_{ij})=\delta_i$. Then for each $1\le j\le \mu$, the point $({\bf x}_{p_j},{\bf x}_{q_j})$ lies on an irreducible component of the zero set of the ideal $\mathcal{I}=\mathcal{I}_{\delta_1}+\ldots+\mathcal{I}_{\delta_k}$ of dimension at most one. 

By B\'{e}zout's Theorem, the number of zero-dimensional components of $Z_\complex(I)$ is $(d,m,\deg D)$-bounded. Thus, for sufficiently large $\mu$, not all of the (distinct) points $({\bf x}_{p_j},{\bf x}_{q_j})$ can be zero-dimensional components.

By Lemma~\ref{lemma:difftopert}, whenever $\xi_{pq}(\tau)=\delta$ we may perturb $p$ along $\Gamma\subset\reals^d$ and get a unique perturbed $q\in\Gamma$ (for a perturbed $\tau$) while preserving this equation. It therefore follows that if some point $({\bf x}_{p_j},{\bf x}_{q_j})$ lies on a one-dimensional component of $Z_\complex(\mathcal{I})$, then we may perturb the points $p_j$ and $q_j$ along the curve $\Gamma\subset\reals^d$ while preserving $\xi_{p_j,q_j}(\tau_{ij})=\delta_i$ for all $1\le i\le k$ for appropriately perturbed $\tau_{ij}$.  
\end{proof}

The following analogues of Propositions~\ref{prop:rigiddist} and \ref{prop:rigidzar} are proved in Section~\ref{sect:rig}.

\begin{proposition}\label{prop:rigstrgen} 
Let $(p,q)\in P^{2*}$ and suppose that the there exists a triangular framework based at $\{p,q\}$ which is not $D$-infinitesimally flexible on $\Gamma$. 

Then, $I$ may be partitioned into a $(d,m,\deg D)$-bounded number of intervals with non-empty interiors \[I=\bigcup_r I_r\] such that whenever $\tau_1, \tau_2\in I$ are distinct points and the $(2,2)$-framework \[\mathcal{K}_\Gamma(\{p,q\},\{\gamma(\tau_1),\gamma(\tau_2)\})\] is $D$-infinitesimally flexible along $\Gamma$, the points $\tau_1$ and $\tau_2$ do not lie in the same $I_r$.

In particular, for $k$ sufficiently large, depending only on $d$ and $m$, and any distinct points $t_1, \ldots, t_k\in\Gamma$, the $(2,k)$-framework \[\mathcal{K}_\Gamma(\{p,q\},\{t_1, \ldots, t_k\})\] is not $D$-infinitesimally flexible on $\Gamma$.
\end{proposition}

\begin{proposition}\label{prop:rigidzargen} 
Suppose that $\Gamma$ is not $(D,\mathcal{T})$-degenerate. Then there exists $P_0\subset P$ such that $|P_0|\gtrsim_{d,m}|P|$ and for each pair $(p,q)\in P_0^{2*}$ there exists a triangular framwork based at $\{p,q\}$ which is not $D$-infinitesimally flexible on $\Gamma$. 
\end{proposition}

By replacing $P$ with $P_0$ and arguing as in the rational curves case, we obtain Theorem~\ref{thm:general} for real algebraic curves.

\section{Proof of rigidity results}\label{sect:rig}

In this section, we prove Propositions~\ref{prop:rigiddist}, \ref{prop:rigidzar}, \ref{prop:rigstrgen} and \ref{prop:rigidzargen}.

Let $\Gamma$ be a curve with injective singularity-free analytic parametrization $\gamma:I\to\reals^d$ where $I\subset\reals$ is an open interval. We assume in the sequel that $\Gamma$ is a real algebraic curve of algebraic degree $m$.

Write $D=D(X,Y)$ for $X,Y\in\reals^d$. Suppose that $[D,\Gamma]$ is simple.
 
\begin{comment}
The main result of this section is the following generalization of Propositions~\ref{prop:rigiddist} and \ref{prop:rigstrgen}.

\begin{proposition}\label{prop:triangleflex}
Let $p,q\in\Gamma$ be distinct points and suppose that there exists a triangular framework based at $\{p,q\}$ which is not $D$-infinitesimally flexible on $\Gamma$. There exists a positive integer $N$ depending on $m$, $d$ and $\deg D$ (but not on $p$ and $q$) with the following property:

The interval $I$ may be partitioned into $N$ intervals with non-empty interiors \[I=\bigcup_{r=1}^{N} I_r\] such that whenever $\tau_1, \tau_2\in I$ are distinct points and the $(2,2)$-framework \[\mathcal{K}_\Gamma(\{p,q\},\{\gamma(\tau_1),\gamma(\tau_2)\})\] is $D$-infinitesimally flexible along $\Gamma$, the points $\tau_1$ and $\tau_2$ do not lie in the same $I_r$.

In particular, for $k$ sufficiently large, depending only on $N$, and any distinct points $t_1, \ldots, t_k\in\Gamma$, the $(2,k)$-framework \[\mathcal{K}_\Gamma(\{p,q\},\{t_1, \ldots, t_k\})\] is not $D$-infinitesimally flexible on $\Gamma$.

Furthermore, if $\gamma$ is a rational parametrization then $N$ can be chosen to depend only on $\deg\gamma$ and $\deg D$.
\end{proposition}
\end{comment}

It will be convenient to construct a suitable analytic function obtained by considering a suitable differential equation which captures the rigidity in our setup; this will allow us to extract suitable bounds. We firstly perform this construction before proceeding to prove our rigidity results.

Since $[D,\Gamma]$ is simple, any $(2,1)$-framework on $\Gamma$ is $D$-flexible. If $U\subset I$ is an open interval, $\bar\tau: U\to I$, $\bar\beta:U\to I$ are smooth and $d_1,d_2\ge 0$ is fixed such that 
\[D(\gamma(\bar\tau(\alpha)),\gamma(\alpha))=d_1\] \[D(\gamma(\bar\tau(\alpha)),\gamma(\bar\beta(\alpha)))=d_2\] for all $\alpha\in U$, 
then differentiating with respect to $\alpha$ yields \[\bar\tau'\dot\gamma(\bar\tau)\cdot D_X(\gamma(\bar\tau),\gamma(\alpha))+\dot\gamma(\alpha)\cdot D_Y(\gamma(\bar\tau),\gamma(\alpha))=0\]
\[\bar\tau'\dot\gamma(\bar\tau)\cdot D_X(\gamma(\bar\tau),\gamma(\bar\beta))+\bar\beta'\dot\gamma(\bar\beta)\cdot D_Y(\gamma(\bar\tau),\gamma(\bar\beta))=0\]
 where \[D_X=\left(\frac{\partial D}{\partial X_1},\ldots,\frac{\partial D}{\partial X_d}\right), \,\,D_Y=\left(\frac{\partial D}{\partial Y_1},\ldots,\frac{\partial D}{\partial Y_d}\right).\]

Eliminating $\bar\tau'$ yields the differential equation
\begin{eqnarray}\label{diffequD}&\bar\beta'\big(\dot\gamma(\bar\beta)\cdot D_Y(\gamma(\bar\tau),\gamma(\bar\beta))\big)\big(\dot\gamma(\bar\tau)\cdot D_X(\gamma(\bar\tau),\gamma(\alpha))\big)& \\
-&\big(\dot\gamma(\alpha)\cdot D_Y(\gamma(\bar\tau),\gamma(\alpha))\big)\big(\dot\gamma(\bar\tau)\cdot D_X(\gamma(\bar\tau),\gamma(\bar\beta))\big)&=0. \nonumber\end{eqnarray}

For each $\alpha, \beta\in I$, define the meromorphic function $\mathcal{H}_{\alpha\beta}=\mathcal{H}^D_{\alpha\beta}$ on $I$ by 
\[\mathcal{H}_{\alpha\beta}(\tau):=\frac{\big(\dot\gamma(\beta)\cdot D_Y(\gamma(\tau),\gamma(\beta))\big)\big(\dot\gamma(\tau)\cdot D_X(\gamma(\tau),\gamma(\alpha))\big)}{\big(\dot\gamma(\alpha)\cdot D_Y(\gamma(\tau),\gamma(\alpha))\big)\big(\dot\gamma(\tau)\cdot D_X(\gamma(\tau),\gamma(\beta))\big)}.\] Potential singularities at $\tau=\alpha, \beta$ may be removed by setting 
\[
\mathcal{H}_{\alpha\beta}(\alpha)=\mathcal{H}_{\alpha\beta}(\beta)=\frac{\dot\gamma(\beta)\cdot D_Y(\gamma(\alpha),\gamma(\beta))}{\dot\gamma(\alpha)\cdot D_X(\gamma(\alpha),\gamma(\beta))}
\]
Since $[D, \Gamma]$ is simple, it follows that \[\mathcal{H}_{\alpha\beta}(\tau)\ne 0, \infty\] for all $\alpha,\beta,\tau\in I$.

Then the differential equation (\ref{diffequD}) is equivalent to \begin{equation}\label{diffH}\mathcal{H}_{\alpha\,\bar\beta(\alpha)}(\tau)=\frac{1}{\bar\beta'(\alpha)}.\end{equation}

For each $\alpha, \beta\in I$, the derivative $\mathcal{H}'_{\alpha\beta}$ has isolated zeroes on $I$ or it vanishes identically and $\mathcal{H}_{\alpha\beta}$ is constant. 

\begin{lemma}\label{Hconsttriangle}
If $\mathcal{H}_{\alpha\beta}$ is constant then any triangular framework based at $\{\gamma(\alpha), \gamma(\beta)\}$ is $D$-infinitesimally flexible on $\Gamma$. Furthermore, if $\mathcal{H}_{\alpha\beta}$ is constant for every $\alpha,\beta\in I$ then $\Gamma$ is $\mathcal{T}$-degenerate.
\end{lemma}

\begin{proof}
Suppose that $\mathcal{H}_{\alpha\beta}$ is constant. Let $\tau\in I\setminus\{\alpha,\beta\}$. The framework \[\mathcal{K}_\Gamma(\{\gamma(\alpha), \gamma(\beta)\},\{\gamma(\tau)\})\] is $D$-infinitesimally flexible for each $\tau\in I$ (in fact, it is $D$-flexible) and there exist $a_\tau, b_\tau \in\reals$ such that 
\[(a_\tau\dot\gamma(\tau),\dot\gamma(\alpha))\cdot\nabla D(\gamma(\tau),\gamma(\alpha))=(a_\tau\dot\gamma(\tau), b_\tau\dot\gamma(\beta))\cdot\nabla D(\gamma(\tau),\gamma(\beta))=0.\]

Since $\mathcal{H}_{\alpha\beta}(\tau)=b_\tau^{-1}$, it follows that $b=b_\tau$ is independent of $\tau$. Taking $\tau\to\beta$, it follows that $a_\tau\to b$. Thus, for each $\tau\in I\setminus\{\alpha,\beta\}$, 
\begin{eqnarray}
&(a_\tau\dot\gamma(\tau),\dot\gamma(\alpha))\cdot\nabla D(\gamma(\tau),\gamma(\alpha))&\nonumber\\
=&(a_\tau\dot\gamma(\tau),b\dot\gamma(\beta))\cdot\nabla D(\gamma(\tau),\gamma(\beta))&\nonumber\\
=&(\dot\gamma(\alpha),b\dot\gamma(\beta))\cdot\nabla D(\gamma(\alpha),\gamma(\beta))&=\,0.\nonumber
\end{eqnarray}
In other words, $\mathcal{T}_\Gamma(\{\gamma(\alpha), \gamma(\beta), \gamma(\tau)\})$ is $D$-infinitesimally flexible.

Now suppose that $\mathcal{H}_{\alpha\beta}$ is constant for all $\alpha,\beta\in I$. Write $h(\alpha,\beta)$ for this constant. Observe that \[\frac{1}{h}:I^2\to\reals\] defines a real-analytic function. 

Let $\alpha_0, \tau_0, \beta_0$ be distinct points in $I$ and let $d_1=D(\gamma(\tau_0), \gamma(\alpha_0))$, $d_2=D(\gamma(\tau_0), \gamma(\beta_0))$, $d_3=D(\gamma(\alpha_0), \gamma(\beta_0))$. By perturbing $\alpha_0$, we obtain an open neighbourhood $U_1\subset I$ of $\alpha_0$ and a smooth function $\beta_1:U_1\to\ I$ such that $\beta_1(\alpha_0)=\beta_0$ and \[d_3=D(\gamma(\alpha), \gamma(\beta_1(\alpha)))\] for every $\alpha\in U_1$. Similarly, we obtain an open neighbourhood $U_2\in I$ of $\alpha_0$ and smooth functions $\tau:U_2\to I$, $\beta_2:U_2\to I$ such that $\tau(\alpha_0)=\tau_0$, $\beta_2(\alpha_0)=\beta_0$ and
\[
d_1=D(\gamma(\tau(\alpha)), \gamma(\alpha))
\]
\[
d_2=D(\gamma(\tau(\alpha)), \gamma(\beta_2(\alpha))).
\]

By (\ref{diffH}), \[\beta_1'(\alpha)=\frac{1}{h}(\alpha,\beta_1(\alpha))\]\[\beta_2'(\alpha)=\frac{1}{h}(\alpha,\beta_2(\alpha))\] for all $\alpha$ in a suitable small open neighbourhood $U\subset I$ of $\alpha_0$. Since $\beta_1(\alpha_0)=\beta_2(\alpha_0)=\beta_0$, the Picard-Lindel\"{o}f Theorem on the uniqueness of solutions to first-order equations implies that \[\beta_1(\alpha)\equiv\beta_2(\alpha)\] for $\alpha\in U$.

Therefore, the framework $\mathcal{T}_\Gamma(\{\gamma(\alpha_0), \gamma(\tau_0), \gamma(\beta_0)\})$ is $D$-smoothly flexible on $\Gamma$.
\end{proof}

We now turn to the proof of Propositions~\ref{prop:rigidzar} and \ref{prop:rigidzargen}.

\begin{proof}[Proof of Proposition~\ref{prop:rigidzar}]
Observe that $\mathcal{H}'_{\alpha,\beta}(\tau)$ is a rational function in $\alpha, \beta, \tau$ of degree $\lesssim(\deg D)(\deg\gamma)$. 

Suppose that $\Gamma$ is not $\mathcal{T}$-degenerate. Then there exists a pair of distinct points $\alpha_0, \beta_0\in I$ such that $\mathcal{H}'_{\alpha_0\beta_0}$ does not vanish identically. Let $\tau_0\in I$ be such that \[\mathcal{H}'_{\alpha_0\beta_0}(\tau_0)\ne 0.\] 

Let $\alpha\in I$. If $\mathcal{H}'_{\alpha\beta_\alpha}(\tau_\alpha)\ne 0$ for some $\beta_\alpha, \tau_\alpha\in I$ then the number of $\beta\in I$ such that $\mathcal{H}'_{\alpha\beta}(\tau_\alpha)=0$ is $(\deg D,\deg\gamma)$-bounded. Furthermore, there are at most finitely many $\alpha\in I$ such that $\mathcal{H}'_{\alpha\beta}(\tau)=0$ for all $\beta, \tau\in I$. Let $S\subset I$ denote the set of such $\alpha$. Then there exists a subset $P_0$ of $P\setminus S$ such that $|P_0|\gtrsim_{\deg D,\deg\gamma}|P\setminus S|$ and, for each pair $(\gamma(\alpha),\gamma(\beta))\in P_0^{2*}$, the function $\mathcal{H}'_{\alpha\beta}$ does not vanish identically. Now, $\mathcal{H}'_{\alpha\beta_0}(\tau_0)$ vanishes for a $(\deg D, \deg\gamma)$-bounded number of $\alpha\in I$ so in fact $|S|$ is $(\deg D, \deg\gamma)$-bounded. Consequently $|P_0|\gtrsim_{\deg D, \deg\gamma}|P|$.
\end{proof}

\begin{proof}[Proof of Proposition~\ref{prop:rigidzargen}]
By Remark~\ref{remark:difftoalg}, the differential equation \[\mathcal{H}'_{\alpha,\beta}(\tau)=0\] is equivalent (by clearing denominators) to a system of polynomial equations in the $\reals^d$-coordinates of the triple $(\gamma(\alpha), \gamma(\beta), \gamma(\tau))$ where all the polynomials have $(d, m, \deg D)$-bounded degree. In particular, it defines a Zariski-closed subset $\Theta$ of $\Gamma\times\Gamma\times\Gamma$.

Suppose that \[\mathcal{H}'_{\alpha_0,\beta_0}(\tau_0)\ne 0\] for points $x_0=\gamma(\tau_0), p_0=\gamma(\alpha_0), q_0=\gamma(\beta_0)$ on $\Gamma$. Then $\Theta$ is proper and, by B\'{e}zout's Theorem, has $(d, m, \deg D)$-bounded degree.

Let $p\in\Gamma$. If $\Gamma\times \{p\}\times \Gamma$ is not contained in $\Theta$ then there exists $x(p)\in\Gamma$ such that $\{(x(p),p)\}\times\Gamma$ is not contained in $\Theta$. Since $\{(x(p),p)\}\times\Gamma$ and $\Theta$ may be generated by polynomials of $(d, m, \deg D)$-bounded degree, B\'{e}zout's Theorem implies that they intersect in a $(d, m, \deg D)$-bounded number of points. Hence, for such $p$, the number of $q\in\Gamma$ such that $\Gamma\times\{p\}\times\{q\}$ is contained in $\Theta$ is $(d, m, \deg D)$-bounded.

Now, $\Gamma\times \{p\}\times \Gamma$ can be completely contained in $\Theta$ for at most finitely many $p\in \Gamma$. Let $S$ denote the set of such $p$. There thus exists a subset $P_0$ of $P\setminus S$ such that $|P_0|\gtrsim_{d,m}|P\setminus S|$ and, for each pair $(p,q)\in P_0^{2*}$, \[\mathcal{H}'_{\gamma^{-1}(p)\gamma^{-1}(q)}\] does not vanish identically on $I$.

Since ${x_0}\times\Gamma\times{q'}$ is not completely contained in $\Theta$, it intersects it in a $(d, m, \deg D)$-bounded number of points. But \[{x_0}\times S\times{q'}=(\Gamma\times S\times\Gamma)\cap({x_0}\times\Gamma\times{q'})\subset\Theta\cap({x_0}\times\Gamma\times{q'})\] so $|S|$ is $(d, m, \deg D)$-bounded. Consequently, $|P_0|\gtrsim_{d,m,\deg D}|P|$.
\end{proof}

Finally, we prove Propositions~\ref{prop:rigiddist} and \ref{prop:rigstrgen}.

\begin{proof}[Proof of Propositions~\ref{prop:rigiddist} and \ref{prop:rigstrgen}]
We will firstly show that, in the case when the zeroes of $\mathcal{H}'_{\alpha\beta}$ are isolated, there is a suitably bounded number of them.

In the case where the parametrization $\gamma$ is rational, $\mathcal{H}'_{\alpha\beta}(\tau)$ is a rational function $\tau$ of degree $\lesssim\deg D\deg\gamma$. Thus, it vanishes identically or has at most $\lesssim\deg D\deg\gamma$ zeroes.

In the case where the curve $\Gamma$ is a real algebraic curve, there is an analogous bound. By Remark~\ref{remark:difftoalg}, the differential equation $\mathcal{H}'_{\alpha\beta}(\tau)=0$ is equivalent to a system of polynomial equations in the coordinates of $\gamma(\tau), \gamma(\alpha), \gamma(\beta)$ with $(d, m, \deg D)$-bounded degrees. If $\mathcal{H}'_{\alpha\beta}$ does not vanish identically then the points $\gamma(\tau)\in \Gamma$ such that $\mathcal{H}'_{\alpha\beta}(\tau)=0$ form a Zariski-closed proper subset of $\Gamma$ consisting of finitely many points. By B\'{e}zout's Theorem, this subset has $(d, m, \deg D)$-bounded cardinality with a bound independent of $\alpha, \beta$.

Let $\alpha, \beta\in I$ be such that $\mathcal{H}'_{\alpha\beta}$ does not vanish identically. Partition the interval $I$ into the smallest number $N_\Gamma$ of intervals \[I=\bigcup_{r=1}^{N_\Gamma}I_r\] such that $\mathcal{H}'_{\alpha\beta}(\tau)\ne 0$ for all $\tau$ in the interior of each $I_r$.

Then, for each $r$ and $\tau_1, \tau_2\in I_r$ distinct from $\alpha, \beta$ such that $\tau_1<\tau_2$, the $(2,2)$-framework \[\mathcal{K}(\{\gamma(\alpha),\gamma(\beta)\},\{\gamma(\tau_1),\gamma(\tau_2)\})\] is not $D$-infinitesimally flexible. Indeed, if it is then (\ref{diffH}) implies that \[\mathcal{H}_{\alpha\beta}(\tau_1)=\mathcal{H}_{\alpha\beta}(\tau_2).\] Therefore $\mathcal{H}'_{\alpha\beta}$ has a zero on the interval $(\tau_1, \tau_2)$, but this cannot happen by construction.

Using the bounds on the number of zeroes of $\mathcal{H}'_{\alpha\beta}$ above, completes the proof.

\end{proof}

\section{Reduction to simplicity for the distance-squared function}\label{sect:simplicity}

We will now restrict our attention to the square distance function on $\reals^d$ given by \[D(x,y)=\norm{x-y}^2\] for $x,y\in\reals^d$ and show how, given a general real algebraic curve $\Gamma$, we may reduce to the situation where $[D,\Gamma]$ is simple (so that Theorem~\ref{thm:general} applies).

\subsection{Rationally parametrized curves}\label{sect:simplerat}

Assume first that $\Gamma$ has a rational parametrization $\gamma:I\to\reals^d$; this case is elementary. The case of general real algebraic curves is dealt with in the next section using more sophisticated tools.

Without loss of generality, $\Gamma$ is not a straight line and does not lie in an affine hyperplane. Indeed, suppose this is not the case and $\Gamma$ lies in a $d'$-dimensional affine subspace of $\reals^d$ but does not lie in any $d''$-dimensional affine subspace for $d''<d'$. The cardinality of $\Delta(P)$ is invariant under rigid motions, so in proving Theorem~\ref{thm:algdist}, it is no loss of generality to assume that this affine subspace is equal to $\reals^{d'}\times\{0\}^{d-d'}$ for some $d'<d$. If $d'=1$ then $\Gamma$ is the open subset of a line and there is nothing to prove. If $d'>1$, the subsequent discussion then applies with $d'$ replacing $d$. 

Furthermore, we may partition $I$ into the union of $N$ disjoint open intervals $\{I_j\}_{j=1}^N$ and a finite set of exceptional points, as \[I=E\cup\bigcup_{j=1}^NI_j,\] with the property that for each $1\le j\le N$, $\gamma|_{I_j}:I_j\to\reals^d$ is injective with non-vanishing first derivative and it defines a curve, $\Gamma_j$, such that whenever $\alpha,\beta\in I_j$ are distinct points, any affine hyperplane which is orthogonal to $(\gamma(\alpha)-\gamma(\beta))$ intersects $\Gamma_j$ in at most one point. Moreover, the partitioning may be performed with $\deg\gamma$-bounded $N$ and $|E|$. Indeed, since $\Gamma$ does not lie in an affine hyperplane, no rational component of $\dot\gamma(t)$ is identically zero. Therefore the set $E$ of $t\in I$ such that any component of $\dot\gamma(t)$ vanishes is $\deg\gamma$-bounded. We may partition $\reals$ as \[\reals=E'\cup\bigcup_{j=1}^N I_j\] into the union of these exceptional points and a finite number of open intervals, $\{I_j\}_{j=1}^N$, whose number is $\deg\gamma$-bounded, such that on each $I_j$ none of the rational components of $\dot\gamma(t)$ vanish. Then $\gamma|_{I_j}$ is certainly injective; in fact, each of its components is strictly monotone. Furthermore, if $p_1,p_2\in \Gamma_j$ are distinct points and there are two points $q_1,q_2\in\Gamma_j$ lying on an affine hyperplane orthogonal to $(p_2-p_1)$, it follows that $(p_2-p_1)$ is orthogonal to $(q_2-q_1)$. If we express these vectors in Euclidean coordinates, \[(p_2-p_1)=(a_1,\ldots,a_d)\]\[(q_2-q_1)=(b_1,\ldots,b_d)\] then, by the strict monotonicity of each component of $\gamma|_{I_j}(t)$, each $a_j$ is non-zero and the products $a_1b_1,\ldots,a_db_d$ are either all non-negative or all non-positive. But \[0=(p_2-p_1)\cdot(q_2-q_1)=\sum_j a_jb_j,\] which forces $b_j=0$ for all $j$, i.e. $q_1=q_2$.

By replacing $P$ with $P\setminus E$ and $\Gamma$ with a curve $\Gamma_j$ which contains $\gtrsim_{\deg\gamma}|P|$ points of $P\setminus E$, in proving Theorem~\ref{thm:algdist}, we may thus assume without loss of generality (up to the loss of a constant factor depending only on the degree of $\gamma$) that $\gamma$ itself is injective, each component of $\dot\gamma(t)$ does not vanish and $\Gamma$ has the property that whenever $\alpha,\beta\in I$ are distinct points, any affine hyperplane which is orthogonal to $(\gamma(\alpha)-\gamma(\beta))$ intersects $\Gamma$ in at most one point. 

With this reduction, $[D,\Gamma]$ is simple: The injectivity condition in the definition is satisfied because if $t\in I$ satisfies \[\norm{\gamma(t)-\gamma(\alpha)}=d_1\]\[\norm{\gamma(t)-\gamma(\beta)}=d_2\] for some distinct $\alpha,\beta\in I$ and $d_1,d_2\ge 0$ then $\gamma(t)$ lies on the intersection of two hyperspheres centered at $\gamma(\alpha)$ and $\gamma(\beta)$. Therefore, it lies on a certain affine hyperplane which is orthogonal to the vector $(\gamma(\beta)-\gamma(\alpha))$. So $t$ is uniquely determined. 

The submersion condition is also satisfied. Indeed, the derivative of $\norm{\gamma(t)-\gamma(\alpha)}^2$ with respect to $t$ is 
\[
2\dot\gamma(t)\cdot(\gamma(t)-\gamma(\alpha))
\] and, expressing $\dot\gamma(t)$, $(\gamma(t)-\gamma(\alpha))$ in Euclidean coordinates \[\dot\gamma(t)=(a_1,\ldots,a_d)\] \[(\gamma(t)-\gamma(\alpha))=(b_1,\ldots,b_d),\] the scalar product \[\dot\gamma(t)\cdot(\gamma(t)-\gamma(\alpha))=\sum_ja_jb_j\] is non-zero since the strict monotonicity of each component of $\gamma(t)$ implies that none of the coordinates $a_j$, $b_j$ are zero and the products $a_1b_1,\ldots, a_db_d$ are either all positive or all negative. 

\subsection{Real algebraic curves}\label{sect:simplealg}

Let $\Gamma\subset\reals^d$ be a real algebraic curve of (geometric) degree $n$ and algebraic degree $m$ which does not lie in an affine hyperplane. By considering the irreducible component of $\Gamma$ containing the most points of $P$ and losing a constant factor depending on $n$, we may assume in proving Theorem~\ref{thm:algdist} that $\Gamma$ is irreducible. Let $\mathcal{I}_{X_1,\ldots,X_d}\subset\complex[X_1, \ldots, X_d]$ be a prime ideal generating $\Gamma$ which has the property that there is a generating set $\{f_1, \ldots, f_r\}$ for $I$ such that $\deg f_j\le m$ for $j=1, \ldots, r$.  

Let $S$ be the set of singularities which lie on $\Gamma$. By B\'{e}zout's Theorem, the cardinality of $S$ is $(d,m)$-bounded. Consider the equivalence relation $\sim$ on the non-singular points of $\Gamma\subset\reals^d$ where $p\sim q$ whenever there is a continuous path from $p$ to $q$ along $\Gamma$ which does not cross any points of $S$. By the Thom-Milnor Theorem, $\Gamma\subset\reals^d$ consists of a $(d,m)$-bounded number of connected components. By considering the number of points of intersection between $\Gamma$ and a $d$-bounded number of suitably chosen affine hyperplanes near each singularity, it then follows that the number of equivalence classes is $(d,n,m)$-bounded. Since $n\le m^d$, the number of classes is, in fact, $(d,m)$-bounded. Each class is a connected, singularity-free open subset of $\Gamma$ and therefore, by the Implicit Function Theorem, has a real-analytic parametrization $\gamma:I\to\reals^d$ for some open interval $I\subset\reals$ which covers the entire equivalence class except for possibly one exceptional point. By choosing the equivalence class with the most points of $P$ and losing a $(d,m)$-bounded factor, we may thus assume that $\Gamma$ itself has an analytic parametrization $\gamma:I\to\reals^d$. We may assume, as above, that $\Gamma$ does not lie in an affine hyperplane. Then, none of the components of $\dot\gamma$ vanish identically and, by Remark~\ref{remark:difftoalg}, we may subdivide $I$ appropriately into a $(d,m)$-bounded number of open intervals and exceptional points, similarly to the rational curves case, and thus reduce to the case where $[\norm{\cdot}^2,\Gamma]$ is simple.

\section{Which curves are $\mathcal{T}$-degenerate?}\label{sect:degen}

Let $D(X,Y)=\norm{X-Y}^2$. Suppose that $\Gamma\subset\reals^d$ has a real-analytic singularity-free parametrization $\gamma:I\to\reals^d$, $[D,\Gamma]$ is simple and $\Gamma$ is $\mathcal{T}$-degenerate. We may assume, without loss of generality, that $\gamma$ is a unit-speed parametrization.

For each $d\ge 0$ and $\tau\in I$, let $\alpha(\tau,d)\in I$ be the least $\alpha(\tau,d)\ge\tau$ such that \[\norm{\gamma(\alpha(\tau,d))-\gamma(\tau)}=d\] when such an $\alpha(\tau,d)$ exists. We extend the definition of $\alpha(\tau,d)$ to negative $d$: we define $\alpha(\tau,d)$ for $d<0$ to be the largest $\alpha(\tau,d)\le\tau$ such that \[\norm{\gamma(\alpha(\tau,d))-\gamma(\tau)}=|d|.\]

Fix $\tau_0\in I$. Without loss of generality, we may take $\tau_0=0$. There exists a sufficiently small $\delta>0$ such that $\alpha(\tau,d)$ is defined for all $|\tau|, |d| <\delta$. Observe that $\alpha(\tau,d)$ is real analytic for $\tau\in(-\delta,\delta)$ for each fixed $|d|<\delta$. 

For each fixed $|\tau|<\delta$, $d\mapsto\alpha(\tau,d)$ is continuous. Furthermore, the derivative $\partial_2\alpha(\tau,d)$ with respect to $d$, agrees with the continuous function \[\frac{2d}{\dot\gamma(\alpha(\tau,d))\cdot\big(\gamma(\alpha(\tau,d))-\gamma(\tau)\big)}\] whenever $d\ne 0$. Since \[\text{lim}_{d\to 0}\,\frac{2d}{\dot\gamma(\alpha(\tau,d))\cdot\big(\gamma(\alpha(\tau,d))-\gamma(\tau)\big)}=\text{lim}_{d\to 0}\,\frac{2d}{d\dot\gamma(\tau)\cdot\dot\gamma(\tau)}=2,\] it follows that $\partial_2\alpha(\tau,d)$ agrees with a continuous function for all $|d|<\delta$. Moreover, using the analytic equation 
\begin{equation}\label{eqn:normgamma}\norm{\gamma(\alpha(\tau,d))-\gamma(\tau)}^2-d^2=0\end{equation} and the Inverse Function Theorem, we may extend the function $d\mapsto\alpha(\tau, d)$ to a continuous function on a suitably small domain containing $0$ in $\complex$ with these properties remaining valid on this domain. 

Thus, $\alpha(\tau,d)$  is separately analytic and continuous on a small domain, so it is in fact jointly analytic in $\tau$ and $d$. The equation (\ref{eqn:normgamma}) then implies that the power series for $\alpha$ takes the form 
\[\alpha(\tau,d)=\tau+d+d^Nf(\tau,d)\]
for some positive integer $N>1$ and an analytic function $f$ such that $f(\tau,0)$ does not vanish identically.

\begin{comment}

For small $\tau$, writing $\alpha_d:=\alpha(\tau,d)$, we may express the second derivative $\ddot\gamma(\tau)$ as the limit
\[\ddot\gamma(\tau)=\text{lim}_{d\to 0}\,\,2\frac{\gamma(\alpha_d)-\gamma(\tau)}{(\alpha_d-\tau)(\alpha_d-\alpha_{-d})}-2\frac{\gamma(\tau)-\gamma(\alpha_{-d})}{(\tau-\alpha_{-d})(\alpha_d-\alpha_{-d})}.\] Writing \[F_1=1+d^{N-1}f(\tau,d),\] \[F_{-1}=1+(-d)^{N-1}f(\tau,-d)\] and \[F_0=2+d^{N-1}(f(\tau,d)-(-1)^Nf(\tau,-d)),\] it follows that
\[\norm{\ddot\gamma(\tau)}^2=\text{lim}_{d\to 0}\,\,2\frac{d^2}{d^4F_1^2F_0^2}+2\frac{d^2}{d^4F_{-1}^2F_0^2}
+4\frac{d^2\cos\theta_d}{d^4F_{1}F_{-1}F_0^2},\] where \[\cos\theta_d:=\frac{1}{d^2}(\gamma(\alpha_d)-\gamma(\tau))\cdot(\gamma(\alpha_{-d})-\gamma(\tau))\] is independent of $\tau$, since $\Gamma$ is $\mathcal{T}$-degenerate.

Thus, \[\norm{\ddot\gamma(\tau)}^2=\text{lim}_{d\to 0}\,\frac{1+\cos\theta_d}{d^2}.\] Therefore the norm of $\ddot\gamma(\tau)$ is independent of $\tau$ for small $\tau$.

More generally, we have the following result.
\end{comment}

We will show that, in the degenerate case, every derivative of $\gamma$ has constant norm by approximating its value using a finite difference method. As a first step, we observe that the assumption that every triangle may be moved along $\Gamma$ while preserving the edge lengths implies that, in fact, any vertex embedding of any complete graph into $\Gamma$ may be moved.

\begin{lemma}\label{lemma:kndegfromt}
The curve $\Gamma$ is $\mathcal{K}_N$-degenerate for every $N\ge 1$. 
\end{lemma}

\begin{proof}
We induct on $N$. The cases $N=1,2,3$ are follow from our assumptions on $\Gamma$ so we assume that $N\ge 3$.

Let $\mathcal{F}$ be a smoothly flexible $\mathcal{K}_N$-framework on $\Gamma$ with smooth motion $\phi_t$ and let $p\in\Gamma$ be a point which is not a vertex of $\mathcal{F}$. For each vertex $q\in\Gamma$ of the framework, write $v(q)$ for the corresponding vertex of the underlying graph $\mathcal{K}_N$.

Choose any two distinct vertices $q,r$ of $\mathcal{F}$. The triangle with vertices at $p,q,r$ is smoothly flexible with smooth motion $\psi_t$ (for $t$ in a suitable interval), say, by the assumption that $\Gamma$ is $\mathcal{T}$-degenerate. We will identify the vertices corresponding to $q$ and $r$ in the triangular graph with $v(q)$ and $v(r)$. Write $v_p$ for the vertex in the triangular graph corresponding to $p$.

The distances between $p$ and $r$ and between $p$ and $q$ remain constant throughout the smooth motion $\psi_t$. The uniqueness in the motions (see Remark~\ref{rem:uniquedist}) implies that we may choose $\psi_t$ so that $\psi_t|_{\{v(q),v(r)\}}$ is always equal to $\phi_t|_{\{v(q),v(r)\}}$. Informally, this means that the motion of $\mathcal{F}$ and that of the triangle with vertices $p,q,r$ that we are considering match for $q$ and $r$. 

For each $t$, the distances from any vertex $s\in\Gamma$ of $\mathcal{F}$ to $q$ and $r$ remain fixed (as $\phi_t$ is a motion) and the distances from $p$ to $q$, $q$ to $r$ and $r$ to $p$ also remain fixed (as $\psi$ is a motion), so the distance from $s$ to $p$ also remains fixed. Therefore, each distance from $p$ to a vertex of $\mathcal{F}$ remains fixed as $t$ varies. Thus, $\phi_t$ extends to a smooth motion of the $\mathcal{K}_{N+1}$-framework obtained by adding $p$ to the vertices of $\mathcal{F}$.

\end{proof}

\begin{lemma}\label{lemma:constcurv}
For each $k\ge 1$, the norm $\norm{\gamma^{(k)}}$ of the $k$-th derivative of $\gamma$ is constant.
\end{lemma}

\begin{proof}
For sufficiently small $\tau$ and $d>0$, we may approximate $\gamma^{(k)}(\tau)$ by a finite difference approximation sampled at the points $\alpha_{jd}:=\alpha(\tau,jd)$ for $-m_1\le j\le m_2$ where $m_1, m_2$ are positive integers such that $m_1+m_2=k$. This leads to an expression of the form
\[
\gamma^{(k)}(\tau)=\text{lim}_{d\to 0}\,\,\sum_{-m_1\le j< m_2}C_j\frac{\gamma(\alpha_{jd})-\gamma(\alpha_{(j+1)d})}{D_j(\tau,d)}
\]
where each $C_j$ is a constant and each denominator $D_j(\tau,d)$ is the product of $k$ (possibly repeated) factors of the form $(\alpha_{nd}-\alpha_{n'd})$ for various integers $-m_1\le n,n'\le m_2$ such that $n\ne n'$.

Thus,
\[
\norm{\gamma^{(k)}(\tau)}^2=\text{lim}_{d\to 0}\,\,\sum_{-m_1\le i, j< m_2}C_iC_j\frac{\big(\gamma(\alpha_{id})-\gamma(\alpha_{(i+1)d})\big)\cdot\big(\gamma(\alpha_{jd})-\gamma(\alpha_{(j+1)d})\big)}{D_i(\tau,d)D_j(\tau,d)}.
\]

By Lemma~\ref{lemma:kndegfromt}, the $\mathcal{K}_{k+1}$-framework with vertices at $\gamma(\alpha(0,jd))$ is smoothly flexible. Consequently, each of the scalar products 
\[S_{ij}(d):=\big(\gamma(\alpha_{id})-\gamma(\alpha_{(i+1)d})\big)\cdot\big(\gamma(\alpha_{jd})-\gamma(\alpha_{(j+1)d})\big)\]
is independent of $\tau$ for small $\tau$.

Furthermore, \[(\alpha_{nd}-\alpha_{n'd})=d(n-n')+d^2G_{nn'}(\tau),\] for an analytic function $G_{nn'}$, which implies that \[\delta_j:=\text{lim}_{d\to 0}\,\,\frac{D_j(\tau,d)}{d^k}\] is finite, non-zero and independent of $\tau$ and $d$.

Therefore,
\[
\norm{\gamma^{(k)}(\tau)}^2=\text{lim}_{d\to 0}\,\,\sum_{-m_1\le i, j< m_2}C_iC_j\frac{S_{ij}(d)}{d^{2k}\delta_i\delta_j}
\]
is independent of $\tau$ for small $\tau$. By analytic continuation, $\norm{\gamma^{(k)}}^2$ is constant on its entire domain.
\end{proof}

D'Angelo and Tyson show in \cite{helices} that any smooth embedded curve in $\reals^d$ such that all its derivatives have constant norm is a generalized helix as in Definition~\ref{def:helix}. 

\begin{corollary}
Suppose that $\Gamma\subset\reals^d$ has a real-analytic singularity-free parametrization $\gamma:I\to\reals^d$ and $[\norm{\cdot}^2,\Gamma]$ is simple. Then, $\Gamma$ is $\mathcal{T}$-degenerate if and only if it is a generalized helix.
\end{corollary}

For the case when the curve $\Gamma$ is a real algebraic curve, we only want to consider generalized helices which are algebraic curves. The following lemma characterizes such helices, thus completing the proof of Theorem~\ref{thm:algdist}.

\begin{lemma}\label{lemma:algebraichelix}
Let $d>0$, $l,k\ge 0$ and $l+2k=d$. Let $I\subset \reals$ be an open interval. Suppose that $\gamma_T:I\to\reals^{2k}$ is given by 
\[
\gamma_T(t)=(\alpha_1\cos \lambda_1t, \alpha_1\sin \lambda_1 t, \ldots, \alpha_k\cos \lambda_kt, \alpha_k\sin \lambda_k t)
\]
for some $\alpha_1, \ldots, \alpha_k, \lambda_1, \ldots, \lambda_k\in \reals\setminus\{0\}$ and $\gamma_L:I\to\reals^l$ is given by 
\[
\gamma_L(t)=tw
\]
for some $w\in\reals^l$.

Then $\gamma:I\to\reals^d$ given by 
\[
\gamma(t)=(\gamma_T(t),\gamma_L(t))\in\reals^{2k}\times\reals^l
\]
parametrizes an open subset of a real algebraic curve if and only if either $l=0$ and for each $1\le i, j\le k$ the ratio $\frac{\lambda_i}{\lambda_j}$ is rational or, alternatively, $k=0$.
\end{lemma}

This lemma is a consequence of the following elementary observation; we prove it here for completeness.

\begin{lemma}\label{lemma:sincossilly}
Let $\rho\in\reals$. There exists a non-zero polynomial $Q_\rho\in\reals[X,Y]$ such that \[Q_\rho(\sin t,\sin \rho t)\equiv 0\] if and only if $\rho\in \mathbb{Q}$.
\end{lemma}

\begin{proof}
For each pair of integers $m$ and $n$, the functions $\sin m\tau$ and $\sin n\tau$ are algebraic over the field $\reals(\sin\tau)$. By considering the resultant eliminating $\sin\tau$ of the minimal polynomials, for example, it follows that there exists a non-zero polynomial $P_{mn}\in\reals[X,Y]$ such that \[P_{mn}(\sin m\tau, \sin n\tau)\equiv 0.\]

If $\rho=\frac{n}{m}\in\mathbb{Q}$, then defining $Q_\rho:=P_{mn}$ gives \[Q_\rho(\sin t,\sin \rho t)=P_{mn}\left({\sin m\frac{t}{m}, \sin n\frac{t}{m}}\right)=0.\]

Conversely, if there exists a non-zero polynomial $Q\in\reals[X,Y]$ such that \[Q(\sin \tau,\sin \rho\tau)\equiv 0\] for some $\rho\in\reals$, then $\sin\rho\tau$ is algebraic over $\reals(\sin\tau)$ so also over $\complex(e^{i\tau})$. Therefore $e^{i\tau}$ is algebraic over $\complex(\sin\rho\tau)$ so also over $\complex(e^{i\rho\tau})$. Thus there exists a non-zero polynomial $H\in\complex[X,Y]$ given by \[H(X,Y)=\sum_{0\le m,n\le M}c_{mn}X^mY^n\]such that \[H(e^{i\tau}, e^{i\rho\tau})\equiv 0.\] Hence 
\begin{equation}\label{vandcoeff}\sum_{0\le m,n\le M}c_{mn}e^{i\tau(m+\rho n)}\equiv 0
\end{equation} for some integer $M$ and coefficients $(c_{mn})_{0\le m,n\le M}$.

 Let $\sigma:\{1, \ldots, (M+1)^2\}\to\{0, 1, \ldots, M\}^2$ be any bijection. Let ${\bf c}$ be the $(M+1)^2\times 1$ non-zero complex column vector whose $\beta$-th entry is \[{\bf c}_\beta=c_{\sigma_1(\beta)\sigma_2(\beta)}.\] Fix a large $K>0$ and let ${\bf A}$ be the $(M+1)^2\times(M+1)^2$ complex matrix whose $(\alpha, \beta)$-th entry is \[{\bf A}_{\alpha\beta}= e^{i(\sigma_1(\beta)+\rho\sigma_2(\beta))\frac{(\alpha -1)}{K}}.\] Then, by considering $\tau=0, \frac{1}{K}, \ldots, \frac{M^2}{K}$ in (\ref{vandcoeff}), it follows that \[{\bf Ac}={\bf 0}.\] Hence ${\bf A}$ is a Vandermonde matrix with a non-trivial kernel so two of the entries in the second row must be equal. Thus there exist integers $0\le m,n,m',n'\le M$ such that $(m,n)\ne(m',n')$ and \[e^{i\frac{(m+\rho n)}{K}}=e^{i\frac{(m'+\rho n')}{K}}.\] Thus 
\[
\frac{(m-m')+\rho (n-n')}{K}\equiv 0 \mod 2\pi
\]
and, by choosing a sufficiently large $K$, we deduce that \[\rho=\frac{m'-m}{n-n'}\in\mathbb{Q}.\]
\end{proof}

\begin{proof}[Proof of Lemma~\ref{lemma:algebraichelix}]
If $k=0$ then $\gamma$ parametrizes a line and this is certainly an algebraic curve. If $k>0$, $l=0$ and each $\frac{\lambda_j}{\lambda_1}=\rho_j\in\mathbb{Q}$, then $\gamma$ parametrizes an open subset of the real algebraic curve given by the $2k-1$ polynomials 
\[
X_j^2+Y_j^2-\alpha_j^2\,\,\,\,\text{ for } j= 1, \ldots,  k
\]
\[
Q_{\rho_j}(\alpha_1^{-1}Y_1,\alpha_j^{-1}Y_j)\,\,\,\,\text{ for } j=2, \ldots,  k
\]
in $\reals[X_1, Y_1, X_2, Y_2, \ldots, X_k, Y_k]$, where $Q_\rho$ is the polynomial from Lemma~\ref{lemma:sincossilly}. 

Conversely, suppose that $\gamma$ parametrizes a non-empty open subset of a real algebraic curve, $\Gamma$. For each distinct $1\le i,j\le d$, write \[\pi_{i,j}:\reals^d\to\reals^2\] for the projection onto the $i$-th and $j$-th coordinates. Then the Zariski-closure of $\pi_{i,j}(\Gamma)$ is at most one-dimensional. 

Suppose for contradiction that $k>0$ and $l>0$. Then \[\pi_{2, (2k+r)}(\gamma(t))=(\alpha_1\sin\lambda_1 t, tw_r),\] where $1\le r\le l$ is chosen so that $w_r\ne 0$. Thus, $\pi_{2,(2k+r)}(\Gamma)$ intersects the line $\{0\}\times\reals\subset\reals^2$ at infinitely many points of the form $(0,w_r\frac{2\pi n}{\lambda_1})$ for $n\in\mathbb{Z}$. This contradicts the fact that the Zariski-closure of $\pi_{2,(2k+r)}(\Gamma)$ is at most one-dimensional.

If $l=0$ then for each $2\le j\le k$, \[\pi_{2,2j}(\gamma(t))=(\alpha_1\sin(\lambda_1 t), \alpha_j\sin(\lambda_jt)).\] Since this is a one-dimensional algebraic curve in $\reals^2$, it follows by Lemma~\ref{lemma:sincossilly} that $\frac{\lambda_j}{\lambda_1}\in\mathbb{Q}$.
\end{proof}

\section{Pinned triangle areas}\label{sect:areas}

To illustrate how the same method can be used for other quantites of interest, we now prove Theorem~\ref{thm:algareas}.

We consider the case when $d=2$ and \[D(x,y)=(x\times y)^2 := (x_1y_2-x_2y_1)^2\] for $x,y\in\reals^2$. Then $\frac{1}{2} |x\times y|=\frac{1}{2}D(x,y)^{1/2}$ is the area of the triangle with vertices at $x$, $y$ and the origin. Note that $D(x,y)=0$ if and only if $x$ is parallel to $y$.

\subsection{Reduction to simplicity}

Let $\Gamma$ be a rational curve which does not pass through the origin with rational parametrization $\gamma:I\to \reals^2$ for an open interval $I$. Write $\gamma(t)=(x(t),y(t))$. We assume that $\Gamma$ is not a straight line or a hyperbola centered at the origin or else there is nothing to prove.

Similarly to what was done in Section~\ref{sect:simplicity}, by losing a constant factor depending only on $\deg \gamma$, we may assume that the rational functions $x(t)$, $\dot x(t)$, $\frac{d}{dt}\big(\frac{y(t)}{x(t)}\big)$, $\frac{d}{dt}\big(\frac{\dot y(t)}{\dot x(t)}+\frac{y(t)}{x(t)}\big)$ do not vanish on $I$. Indeed, none of them vanish identically since $\Gamma$ is not a straight line or a hyperbola centered at the origin so the number of $t\in I$ where any of them vanish is $\deg\gamma$-bounded. In particular, observe that these assumptions imply that $\gamma$ is injective and, given any non-zero vector $v\in\reals^2$, there is at most one value $\alpha\in I$ such that $\gamma(\alpha)$ is parallel to $v$.

With these assumptions, we check that $[D,\Gamma]$ is simple: Suppose $t\in I$ satisfies \[(\gamma(t)\times\gamma(\alpha))^2=d_1\]\[(\gamma(t)\times\gamma(\beta))^2=d_2\] for some distinct $\alpha,\beta\in I$ and $d_1,d_2\ge 0$. If either of $d_1$ or $d_2$ is $0$ then $t$ is determined uniquely (as $\alpha$ or $\beta$) since $\gamma(t)$ is parallel to $\gamma(\alpha)$ if and only if $t=\alpha$ and similarly for $\beta$. If $d_1, d_2>0$ then $\gamma(t)$ is parallel to a vector of the form $\pm |d_1|^{-1/2}\gamma(\alpha)\pm |d_2|^{-1/2}\gamma(\beta)$. Since $\gamma(\alpha)$ is not parallel to $\gamma(\beta)$, any vector of the form $\pm|d_1|^{-1/2}\gamma(\alpha)\pm|d_2|^{-1/2}\gamma(\beta)$ is non-zero, so $t$ is uniquely determined. Thus the injectivity condition in the definition is satisfied. 

The submersion condition is also satisfied. Indeed, for $\alpha\ne\beta$,
\[\nabla_{(\alpha,\beta)}(\gamma(\alpha)\times\gamma(\beta))^2=(\gamma(\alpha)\times\gamma(\beta))\big(\dot\gamma(\alpha)\times\gamma(\beta), \gamma(\alpha)\times\dot\gamma(\beta)\big).\]
Suppose for contradiction that this derivative vanishes. Since $\gamma(\alpha)$ is not parallel to $\gamma(\beta)$, it follows that $\gamma(\alpha)$ is parallel to $\dot\gamma(\beta)$ and $\gamma(\beta)$ is parallel to $\dot\gamma(\alpha)$. Setting $z(t)=\frac{\dot y(t)}{\dot x(t)}+\frac{y(t)}{x(t)}$, it follows that $z(\alpha)=z(\beta)$ and so the rational function $\dot z(t)$ vanishes at some $t\in I$; this contradicts our assumptions.

Finally, $D$ is a distance polynomial for $\gamma$ since $\gamma(\alpha)$ is parallel to $\gamma(\beta)$ if and only if $\alpha=\beta$.

Therefore, we have reduced to the situation where $[D,\Gamma]$ is simple in the special case when the curve is rationally parametrized. Similarly to what was done in Section~\ref{sect:simplicity}, we also can reduce to simple $[D,\Gamma]$ in the case of a general real algebraic curve $\Gamma$ in $\reals^2$; we omit the details.

\subsection{Which curves are $(D,\mathcal{T})$-degenerate?}

By Theorem~\ref{thm:general}, it now follows that $\Gamma$ is $(D,\mathcal{T})$-degenerate and it remains to classify such curves to complete the proof.

We assume in this section that $\Gamma\subset\reals^2$ has a real-analytic singularity-free parametrization $\gamma:I\to\reals^2$, $[D,\Gamma]$ is simple and $\Gamma$ is $(D,\mathcal{T})$-degenerate. Write $\gamma(t)=(x(t),y(t))$.

Choose $\alpha, \beta\in I$ be such that $\gamma(\alpha)$ is not parallel to $\gamma(\beta)$. The conclusion of Theorem~\ref{thm:algareas} is invariant under the action of $GL_2(\reals)$ so we may assume that $\gamma(\alpha)=(1,0)$ and $\gamma(\beta)=(0,1)$.

Then, for $\mathcal{H}_{\alpha\beta}^D$ defined as in Section~\ref{sect:rig},  \[\mathcal{H}^D_{\alpha\beta}=\frac{\big(\dot\gamma(\beta)\cdot(-2yx,2x^2)\big)\big(\dot\gamma\cdot (0,2y)\big)}{\big(\dot\gamma(\alpha)\cdot (2y^2,-2xy)\big)\big(\dot\gamma\cdot (2x,0)\big)}=\frac{\big(\dot\gamma(\beta)\cdot(-2y,2x)\big)\big(\dot\gamma\cdot (0,2)\big)}{\big(\dot\gamma(\alpha)\cdot (2y,-2x)\big)\big(\dot\gamma\cdot (2,0)\big)}\] is constant and equal to  \[-\frac{\dot\gamma(\beta)\cdot(0,2)}{\dot\gamma(\alpha)\cdot (2,0)}.\]

Therefore,
\[-y\dot y\dot x(\alpha)\dot x(\beta)+x\dot y\dot x(\alpha)\dot y(\beta) + y\dot x\dot x(\alpha)\dot y(\beta) - x\dot x\dot y(\alpha)\dot y(\beta)\equiv 0\]
and hence
\[-y^2\dot x(\alpha)\dot x(\beta)+2xy\dot x(\alpha)\dot y(\beta) - x^2\dot y(\alpha)\dot y(\beta)=C\]
for some real constant $C$.

Thus $\gamma$ parametrizes an open subset of an ellipse or hyperbola centered at the origin and the proof of Theorem~\ref{thm:algareas} is complete.
\qed

\section{Further remarks}\label{sect:further}

Our results may be interpreted as a statement about the expansion of $D|_{\Gamma\times\Gamma}$ for small (in our case, finite) subsets $P\subset\Gamma$. The conclusions are similar in spirit to results such as \cite{hartli} in finite fields.

In this direction, there is the rather general result of Elekes and Szab\'{o} \cite{elekesszabo}: they consider the question of intersections \[(A\times B\times C)\bigcap V\] where $A$, $B$ and $C$ are finite subsets of varieties of the same dimension, $|A|=|B|=|C|$ and $V$ is a suitable variety. In the context of the problem we are considering, they show in particular that there exists a universal constant $\eta>0$ such that if $\Gamma\subset\reals^d$ is a real algebraic curve, $D:\reals^d\times\reals^d\to\reals$ is a polynomial, $V_D\subset\Gamma\times\Gamma\times\reals$ is the variety \[V_D=\{(x,y,z)\in\Gamma\times\Gamma\times\reals\,\,|\,\,z=D(x,y)\}\] and $P\subset\Gamma$, $Q\subset \Delta(P)$ are finite subsets with cardinality $\sim N$ which satisfy
\[|(P\times P\times Q)\bigcap V_D|\gtrsim N^{2-\eta},\]
then the variety $V_D$ is special in a certain sense; essentially, there is an algebraic group acting in the background and the variety is the image of the graph of its multiplication function.

Our results above can be phrased in a similar form, by an elementary averaging argument: for any $\epsilon>0$, if $[D,\Gamma]$ is simple and
\[|(P\times P\times Q)\bigcap V_D|\gtrsim N^{2-\frac{1}{4}+\epsilon},\] then $\Gamma$ is $(D,\mathcal{T})$-degenerate. 

Nonetheless, the $\eta$ obtained by following the proof in \cite{elekesszabo} directly is less than $(\frac{1}{4}-\epsilon)$. With the notation of \cite{elekesszabo}, for small $\epsilon>0$, \[\eta<\frac{\eta'}{2}\le \frac{3-2(\alpha+\beta)}{2}=\alpha-\frac{1}{2}=\frac{1}{2(2D-1)}-\epsilon\le\frac{1}{2(2(4k)-1)}-\epsilon\le\frac{1}{14}-\epsilon.\]

A few weeks after an initial preprint of this paper was released, Pach and de Zeeuw \cite{pachdezeeuw} improved the exponent in Theorem~\ref{thm:algdist} from $1+\frac{1}{4}$ to $1+\frac{1}{3}$ in the special case where the ambient dimension is $2$ (and the polynomial in question is the square distance function). Their method is mostly algebraic and simpler than our argument for that particular case; when the curve $\Gamma$ is planar, there is no need to appeal to more advanced tools from algebraic geometry such as the Thom-Milnor Theorem or the theory of Gr\"{o}bner bases. 

An obstruction to improving the exponent in our argument is Proposition~\ref{prop:xixi}. In \cite{pachdezeeuw}, a different set of curves is considered instead of the Elekes curves considered here; the analogue using our notation would be to consider the curves given implicitly by \[D(\gamma(t), p) - D(\gamma(s), q) = 0\] for $(t,s)\in\reals^2$. This algebraic problem is more complicated and less directly amenable to analytical tools, but does indicate a natural approach to consider when trying to improve the exponent in Theorem~\ref{thm:algdist} (and its variations) when the ambient dimension is not necessarily $2$.

This set of curves was previously also considered by Sharir, Sheffer and Solymosi in \cite{distlines} who improved the bound of Elekes from \cite{elekes} for the quantitative version of Purdy's Conjecture from $1+\frac{1}{4}$ to $1+\frac{1}{3}$. More precisely, they show that if $ P_1$ and $P_2$ are two sets of $N$ points in the plane so that $P_1$ is contained in a line $L_1$, $P_2$ is contained in a line $L_2$, and $L_1$ and $L_2$ are neither parallel nor orthogonal, then the number of distinct distances determined by the pairs $P_1\times P_2$ is $\gtrsim N^{1 + \frac{1}{3}}$.

In both \cite{distlines} and \cite{pachdezeeuw}, the problem considered is, in fact, a bipartite problem: there are two curves $\Gamma_1, \Gamma_2$ and two finite subsets $P_1, P_2$ and the aim is to obtain a lower bound on the cardinality of the set \[\{D(p,q)\,\,|\,\,p\in P_1,  q\in P_2\}.\] While the argument in Section~\ref{sect:gen} and some of the rigidity results still apply almost verbatim (with appropriate modifications) in this bipartite setting, some of the central rigidity ideas, for example $\mathcal{T}$-degeneracy, do not seem to carry over easily. In order to keep the link to rigidity and the main result Theorem~\ref{thm:general} as clear as possible, we have chosen not to attempt to discuss the bipartite version of the problem here.

\vspace{5mm}
\paragraph{\bf{Acknowledgements}}

The author would like to thank M. Christ for several useful conversations and also S. Brodsky and R. Vianna for a helpful discussion on algebraic and geometric aspects of curves.

The author is especially grateful to the two anonymous referees who gave a long list of incredibly constructive comments; this resulted in a significant improvement to the exposition.

\bibliographystyle{amsplain}{}
\bibliography{ddpcurvesbib}

\end{document}